\documentclass[journal]{IEEEtran}

\usepackage[nospace]{cite}

\usepackage{algpseudocode}
\usepackage{amssymb} 
\usepackage{nicefrac}
\usepackage{xcolor}
\usepackage[utf8]{inputenc} 
\usepackage[T1]{fontenc}    
\usepackage{verbatim} 
\usepackage{hyperref}       
\usepackage{url}            
\usepackage{booktabs}       
\usepackage{amsfonts}       
\usepackage{nicefrac}       
\usepackage{microtype}      
\usepackage{amsmath}
\usepackage{amsthm}
\usepackage{amsfonts}
\usepackage{nicefrac}
\usepackage{enumitem}
\usepackage{graphicx}
\usepackage{subcaption}
\usepackage{mathrsfs}
\usepackage{float}
\usepackage{placeins}
\usepackage{mathtools}
\usepackage{dashrule}
\usepackage{booktabs,multirow,array,longtable,tabularx}
\usepackage{amsmath,amssymb,bm}
\usepackage{xcolor,colortbl}
\usepackage{caption,subcaption}
\usepackage{pdflscape}
\usepackage{threeparttable}
\usepackage[linesnumbered,ruled,vlined]{algorithm2e}

\theoremstyle{plain}
\newcommand{\coloneqq}{\mathrel{\mathop:}=}
\newtheorem{theorem}{Theorem}
\newtheorem{definition}{Definition}

\newtheorem{proposition}[theorem]{Proposition}

\theoremstyle{definition}
\newtheorem{remark}[theorem]{Remark}
\theoremstyle{remark}

\newcommand{\mtrx}[1]{\mathbf{#1}}        

\newcommand{\proj}{\Pi}

\definecolor{wingreen}{HTML}{2E7D32}
\definecolor{losered}{HTML}{C62828}
\definecolor{warnorg}{HTML}{E65100}
\definecolor{hdrblue}{HTML}{1565C0}
\definecolor{lightblue}{HTML}{E3F2FD}
\definecolor{lightgreen}{HTML}{E8F5E9}
\definecolor{lightyellow}{HTML}{FFF8E1}
\definecolor{lightgray}{gray}{0.93}

\newcommand{\Rtot}{R_{\mathrm{tot}}}
\newcolumntype{C}[1]{>{\centering\arraybackslash}p{#1}}
\newcolumntype{R}[1]{>{\raggedleft\arraybackslash}p{#1}}
\newcolumntype{L}[1]{>{\raggedright\arraybackslash}p{#1}}
\newcommand{\cmark}{\checkmark}
\newcommand{\xmark}{\ensuremath{\times}}
\newcommand{\pmark}{\ensuremath{\circ}}
\renewcommand{\arraystretch}{1.25}

\newcommand{\mM}{\mathbf{M}}
\newcommand{\mL}{\mathbf{L}}
\newcommand{\mA}{\mathbf{A}}
\newcommand{\mF}{\mathbf{F}}
\newcommand{\mLb}{\mathbf{L}_{\mathrm{b}}}
\newcommand{\mLa}{\mathbf{L}_{\mathrm{a}}}
\newcommand{\mB}{\mathbf{B}}
\newcommand{\mBa}{\mathbf{B}_{\mathrm{a}}}
\newcommand{\mK}{\mathbf{K}}
\newcommand{\mH}{\mathbf{H}}
\newcommand{\mD}{\mathbf{D}}
\newcommand{\mI}{\mathbf{I}}
\newcommand{\mR}{\mathbf{R}}
\newcommand{\mT}{\mathbf{T}}
\newcommand{\mU}{\mathbf{U}}
\newcommand{\mC}{\mathbf{C}}
\newcommand{\mS}{\mathbf{S}}

\newcommand{\vw}{\mathbf{w}}
\newcommand{\va}{\mathbf{a}}
\newcommand{\vc}{\mathbf{c}}
\newcommand{\vd}{\mathbf{d}}
\newcommand{\ve}{\mathbf{e}}
\newcommand{\vz}{\mathbf{z}}
\newcommand{\vone}{\mathbf{1}}
\newcommand{\vzero}{\mathbf{0}}

\newcommand{\vx}{\mathbf{x}}

\newcommand{\Tr}{\operatorname{Tr}}
\newcommand{\diag}{\operatorname{diag}}
\newcommand{\Real}{\mathbb{R}}
\newcommand{\Sym}{\mathbb{S}}
\newcommand{\Ec}{\mathcal{E}_{\mathrm{c}}}
\newcommand{\Eb}{\mathcal{E}_{\mathrm{b}}}
\newcommand{\Eused}{\mathcal{E}_{\mathrm{used}}}
\newcommand{\Erem}{\mathcal{E}_{\mathrm{rem}}}

\newcommand{\Kcone}{\mathcal{K}}
\newcommand{\Cset}{\mathcal{C}}
\DeclareMathOperator{\svec}{svec}

\newcommand{\mQ}{\mathbf{Q}}

\newcommand{\vb}{\mathbf{b}}
\newcommand{\vq}{\mathbf{q}}
\newcommand{\vr}{\mathbf{r}}
\newcommand{\vs}{\mathbf{s}}
\newcommand{\vu}{\mathbf{u}}       
\newcommand{\vv}{\mathbf{v}}       
\newcommand{\vy}{\mathbf{y}}

\newcommand{\vomega}{\mathbf{s}}

\algrenewcommand\algorithmicrequire{\textbf{Input:}}
\algrenewcommand\algorithmicensure{\textbf{Output:}}
\algnewcommand\algorithmicsetup{\textbf{Setup:}}
\algnewcommand\Setup{\item[\algorithmicsetup]}
\algnewcommand\algorithmicinit{\textbf{Init:}}
\algnewcommand\Init{\item[\algorithmicinit]}

\hypersetup{
  colorlinks=true,
  linkcolor=blue,
  citecolor=blue,
  urlcolor=blue
}

\ifCLASSINFOpdf

\else

\fi

\begin{document}

\title{Budget-Constrained Graph Augmentation for Robust Network Design via Kirchhoff Index Minimization}

\author{Omkar~Bhoite,
        V~Sateeshkrishna~Dhuli, \IEEEmembership{Senior Member, IEEE},
        Stefan~Werner, \IEEEmembership{Fellow, IEEE},
        and~Kimmo~Kansanen, \IEEEmembership{Senior Member, IEEE}
\thanks{O. Bhoite,
        S. Werner,
        and K. Kansanen are with the Department
of Electronic Systems, Norwegian University of Science and Technology (NTNU), Trondheim, 7032 Norway. E-mail: \{omkar.bhoite, stefan.werner, kimmo.kansanen\}@ntnu.no. S. Werner is also affiliated with Department of Information and Communication Engineering, Aalto University, Finland.
        V S Dhuli is with the Department of ECE, SRM University-AP, Amaravati 522240, Andhra Pradesh, India, E-mail: \{dvskrishna.nitw@gmail.com\}.  A conference version of this paper is accepted to EUSIPCO 2026. This work is supported in part by the MoST (MobilitetsLab Stor-Trondheim) project; the PERSEUS project, a European Union’s Horizon 2020 research and innovation program under the Marie Sk{\l}odowska-Curie grant agreement No 101034240.}}

\markboth{Submitted manuscript}%
{Bhoite \MakeLowercase{\textit{et al.}}: Budget-Constrained Graph Augmentation}

\maketitle

\begin{abstract}
Enhancing the robustness of deployed networks against failures and disruptions is critical for reliable operation. This requires deciding which new links to install and how strongly to weight them under limited resources. We study this problem through the Kirchhoff index, or total effective resistance, a spectral measure of global connectivity. The resulting augmentation problem couples discrete candidate-edge selection with continuous weight allocation under heterogeneous per-unit deployment costs, a total budget, and an exact-cardinality constraint. For a fixed weighted base graph, this yields a mixed-integer formulation  and a semidefinite relaxation whose optimum lower-bounds the mixed-integer optimum. We cast the relaxation as a cone program and solve it numerically using a homogeneous self-dual embedding and first-order operator splitting. Feasible discrete designs are recovered through rounding-and-repair procedures and assessed by \emph{a posteriori} gap estimates relative to the numerical semidefinite program (SDP)
benchmark.  As a scalable alternative, we develop an exact-$k$, budget-feasible greedy heuristic built on rank-one Laplacian updates and biharmonic-distance caching, and interpret its progress through a Bellman value-to-go benchmark with a conservative spectral lower bound on the local policy ratio. Experiments on synthetic and real infrastructure networks across graph sizes, budgets, weight distributions, and cost regimes show that, under fixed budgets, distance-proportional costs limit the achievable resistance reduction and shift installed conductance toward shorter links relative to uniform per-unit costs.

\end{abstract}


\IEEEpeerreviewmaketitle

\section{Introduction}

\IEEEPARstart{N}{}etworks model systems of relations and interactions across engineered and natural domains, from infrastructure and communications to biological and social organizations~\cite{newman2003structure}. Network robustness denotes the ability of the underlying graph to retain its functionality under structural constraints, failures, or exogenous perturbations. In practice, transportation systems require alternative paths to sustain throughput under edge failures \cite{yang2018designing}; communication and infrastructure networks motivate graph-based resilience metrics and link-addition strategies \cite{alenazi2015evaluation, wang2014improving}; and brain networks exhibit structural resilience that preserves functional connectivity under perturbations \cite{joyce2013human}. To quantify and improve robustness, prior work has developed graph-theoretic and spectral measures together with optimization methods based on them, including graph and spectral metrics \cite{ellens2013graphmeasuresnetworkrobustness, wang2014improving}, structural augmentation \cite{alenazi2015evaluation}, perturbation-based analysis \cite{ceci2018small, cattai2025robust}, and attack-vulnerability studies \cite{albert2000error}.

Among these measures, the Kirchhoff index (or total effective resistance) has emerged as an important graph metric \cite{devriendt2022effective,Minimizing_ER,alenazi2015comprehensive, wang2014improving, Kooij2023}. Its appeal stems from a robustness-oriented interpretation: many strongly weighted alternative paths bring two nodes effectively close, whereas sparse or weakly weighted connectivity pushes them apart. Unlike the shortest path, effective resistance accounts for the cumulative contribution of all routes between a pair of nodes~\cite{devriendt2022effective}. Formally, the Kirchhoff index equals $n$ times the sum of the reciprocals of the nonzero Laplacian eigenvalues \cite{Minimizing_ER}; this spectral form makes it a natural objective for robustness-oriented design. Direct evaluation through the trace formulation involves a Laplacian pseudoinverse, which can become the computational bottleneck on large graphs \cite{huang2025estimating, dwaraknath2023towards}, motivating scalable formulations.

Prior work on total effective resistance addresses several related problems. For a fixed topology, optimal weight allocation over a prescribed edge set has been characterized without edge augmentation~\cite{Minimizing_ER}. Since optimal link addition is \textsc{NP}-hard~\cite{Kooij2023}, work has focused on tractable methods, including efficient greedy selection~\cite{predari2023greedy} and, for unweighted graphs, greedy algorithms with submodularity-ratio guarantees and near-linear-time gradient-based methods~\cite{zhou2025efficient}. However, effective-resistance edge addition is not submodular in general, and its submodularity ratio can approach zero~\cite{achterberg2025non}. Scalable Kirchhoff-index estimation~\cite{huang2025estimating} and resistance-aware graph rewiring~\cite{Attali2025DynamicTG,black2023understanding} address related questions but not budgeted augmentation of a fixed weighted base graph. In air-transportation networks, routes have been added to reduce total effective resistance, but the added weights are treated as fixed inputs~\cite{yang2018designing}. Among these works, closest to our setting is an SDP formulation that jointly optimizes links and weights under a budget, but designs the network from scratch rather than augmenting a deployed weighted graph~\cite{yang2024robustness}. Table~\ref{tab:prior_work} summarizes the main distinctions.

\begin{table*}[!t]
\centering
\footnotesize
\setlength{\tabcolsep}{4pt}
\renewcommand{\arraystretch}{1.2}
\caption{Comparison with related work. \cmark{}: present, $\times$: absent, and $\circ$: limited. \textbf{Base}: augmentation of a fixed weighted base graph; \textbf{Wt}: candidate-edge weight optimization; \textbf{Cost}: explicit edge-cost model or budget constraint; \textbf{SDP}: semidefinite relaxation; \textbf{Rnd}: rounding; \textbf{Grdy}: greedy method.}
\label{tab:prior_work}
\begin{tabularx}{\textwidth}{@{}lcccccc>{\raggedright\arraybackslash}X@{}}
\toprule
\textbf{Work} & \textbf{Base} & \textbf{Wt} & \textbf{Cost} & \textbf{SDP} & \textbf{Rnd} & \textbf{Grdy} & \textbf{Main setting} \\
\midrule
\textbf{Ours} & \cmark & \cmark & \cmark & \cmark & \cmark & \cmark & Exact-$k$, cost-aware augmentation of a fixed weighted base graph; SDP relaxation, rounding, and cost-aware greedy \\
\cite{Growing_Graphs} & \pmark & \xmark & \pmark & \cmark & \cmark & \cmark & Unweighted cardinality-$k$ edge addition maximizing algebraic connectivity $\lambda_2$ \\
\cite{Minimizing_ER} & \pmark & \cmark & \pmark & \cmark & \xmark & \xmark & Convex/SDP conductance allocation on a fixed edge set; no edge addition  \\
\cite{yang2018designing} & \cmark & \xmark & \xmark & \cmark & \cmark & \cmark & Fixed-base route selection with fixed added-edge weights; cardinality $k$; minimizes total effective resistance \\
\cite{predari2023greedy} & \pmark & \xmark & \xmark & \xmark & \xmark & \cmark & Accelerated greedy unweighted cardinality-$k$ edge insertion minimizing total effective resistance \\
\cite{yang2024robustness} & \xmark & \cmark & \cmark & \cmark & \cmark & \pmark & From-scratch weighted SDP design under a route-cost budget; greedy used only for rounding \\
\cite{zhou2025efficient} & \pmark & \xmark & \xmark & \xmark & \xmark & \cmark & Greedy and gradient-based unweighted cardinality-$k$ edge addition minimizing the Kirchhoff index \\
\bottomrule
\end{tabularx}
\end{table*}

We consider augmentation of a pre-deployed undirected weighted base graph whose existing edges and weights remain fixed. We formulate a mixed-integer (MI) problem that jointly selects exactly $k$ candidate edges and assigns their weights under heterogeneous per-unit deployment costs and a total budget. We then relax the binary variables and use the standard Schur-complement SDP representation of the trace-inverse objective~\cite{Minimizing_ER}, yielding a convex program without an explicit Laplacian pseudoinverse.

The SDP relaxation can be solved by standard interior-point or augmented-Lagrangian methods, but their memory requirements grow rapidly with the semidefinite block. We instead cast the relaxation as a cone program and solve it using the homogeneous self-dual embedding (HSDE) ADMM framework of~\cite{o2016conic}. The relaxation itself lower-bounds the mixed-integer optimum, while the solver returns a finite-tolerance numerical solution. We recover feasible discrete designs through rounding and repair and compare the standard top-$k$ rule with support-based weight re-optimization and greedy recovery on the SDP support.

We also develop a cost-aware greedy augmentation algorithm that selects edges sequentially by their marginal decrease in total effective resistance. Biharmonic-distance caching enables constant-time evaluation of each candidate's marginal gain, while Sherman--Morrison rank-one updates maintain the shifted Laplacian inverse and its square after each addition~\cite{predari2023greedy, golub2013matrix}. The resulting designs are evaluated against the numerical SDP benchmark. We analyze the greedy method using a Bellman value-to-go benchmark over the reachable state space, obtaining a local progress and residual-decay characterization with a conservative spectral lower bound on the local policy ratio.

A preliminary conference version of this work introduced the core idea of budget-constrained graph augmentation, together with its SDP relaxation and a variant of the greedy algorithm~\cite{bhoite2026costaware}. This paper substantially extends that work, and its main contributions are as follows. We formulate exact-$k$ augmentation of a fixed weighted base graph as a joint edge-selection and weight-sizing problem under heterogeneous per-unit costs and a total budget. We derive an SDP relaxation that lower-bounds the mixed-integer optimum, characterize the relaxed optimum by a cost-normalized marginal-benefit proposition, and recover feasible discrete designs by rounding and repair with support-based weight re-optimization. We develop an exact-$k$, budget-feasible greedy heuristic based on rank-one inverse updates and biharmonic-distance caching, and analyze its progress relative to a Bellman value-to-go benchmark with a conservative spectral lower bound on the local policy ratio. We report an extensive experimental study on synthetic and real-world networks, covering solver behavior, the alternative rounding schemes, and the spatial footprint of the added edges. It shows that under fixed budgets distance-dependent costs limit the attainable reduction in total effective resistance and shift the added-edge weight distribution toward shorter links.

The remainder of the paper is organized as follows. Section~\ref{framework} formulates the augmentation problem and its SDP relaxation; Section~\ref{conic_and_adpt_relax} describes the conic solution approach. Section~\ref{budget_greedy} presents the greedy method and its analysis. Sections~\ref{sec:experimental-setup} and \ref{sec:results} present the experimental setup and results, respectively.

\textbf{\textit{Mathematical Notations.}}
Italic letters denote scalars, bold lowercase letters denote vectors, and bold uppercase letters denote matrices. For symmetric matrices $\mA$ and $\mB$, $\mA\succeq\mB$ ($\mA\succ\mB$) means that $\mA-\mB$ is positive semidefinite (positive definite). We write $\Tr(\cdot)$ for the matrix trace, $(\cdot)^\top$ for transpose, and $[\va]_i$ for the $i$th entry of $\va$. The set $\mathbb S^n$ contains the real symmetric $n\times n$ matrices, with $\mathbb S_+^n$ and $\mathbb S_{++}^n$ denoting its positive semidefinite and positive definite subsets, respectively. The sets $\mathbb R_{\ge0}^n$ and $\mathbb R_{>0}^n$ denote the nonnegative and strictly positive orthants. The symbols $\mI$, $\vone$, and $\vzero$ denote the identity matrix, all-ones vector, and all-zeros vector of compatible dimensions.

\section{Problem Formulation and SDP Relaxation}\label{framework}
We consider an undirected base graph $\mathcal G(\mathcal{V},\Eb,\mA_{\mathrm b})$ with node set $\mathcal{V}=\{v_1,\dots,v_n\}$ and base edge set $\Eb=\{e_1,\dots,e_m\}$. Each edge $e$ is an unordered pair $\{v_i, v_j\}$ with $i \neq j$. The symmetric weighted adjacency matrix $\mA_{\mathrm b}$ represents the strengths of existing links and is fixed throughout. Let $\mD_{\mathrm b}$ be the corresponding degree matrix and $\mLb = \mD_{\mathrm b} - \mA_{\mathrm b} \in \mathbb{S}^n_{+}$ the base Laplacian. We select $p$ candidate edges from the $\frac{n(n-1)}{2}-m$ non-base node pairs, forming $\Ec=\{e_{m+1},\dots,e_{m+p}\},$ disjoint from $\Eb$. For each candidate $e=\{v_i,v_j\}$, we fix an arbitrary orientation and define $\va_e\in\mathbb R^n$ by $[\va_e]_i=1$, $[\va_e]_j=-1$, and zero elsewhere. The orientation does not affect $\va_e\va_e^\top$. Each candidate edge has a nonnegative design weight $w_e\ge0$ representing its strength, such as  conductance, bandwidth, or inverse delay.

The augmentation Laplacian is 
\begin{equation}
\mLa(\vw)=\sum_{e\in \Ec} w_e\,\va_e\va_e^\top, \qquad \mLa(\vw) \in \mathbb{S}^{n}_{+}.
\label{eq:La}
\end{equation}
The augmented graph has Laplacian \(\mL(\vw)=\mLb+\mLa(\vw).\) We assume the base graph is connected; hence \(\mM(\vw) \coloneqq \mL(\vw)+\vone\vone^\top/n\succ0\,,\) for every $\vw \ge \vzero$~\cite{sundin2017connectedness}. Its total effective resistance is~\cite{Minimizing_ER}
\begin{equation}
R_{\mathrm{tot}}(\vw)=n\,\operatorname{Tr}\!\Big(\mM^{-1}(\vw)\Big)-n\;.
\end{equation}

Each candidate edge $e$ has a per-unit deployment cost $c_e > 0$, collected in $\vc \in \mathbb R^p$, reflecting factors such as length, terrain, or installation overhead. Given a budget $C>0$, a target cardinality $k\in\{1,\dots,p\}$, and weight bounds $0<w_{\min}\le w_{\max}$, we introduce binary selection variables $\vz\in\{0,1\}^p$ and continuous weights $\vw\in\mathbb R_{\ge0}^p$. The joint edge-selection and weight-sizing problem is
\begin{equation*}
\begin{aligned}
\underset{\vw,\,\vz}{\mathrm{minimize}}\quad &
R_{\mathrm{tot}}(\vw) 
\\
\mathrm{subject\ to}\quad &
\vc^\top \vw \le C, \hspace{120pt}(\mathcal{P}0)\\
& z_i w_{\min} \le w_i \le z_i w_{\max},\quad i=1,\dots,p,\\
& \vone^\top \vz = k,\qquad z_i\in\{0,1\}
\end{aligned}
\end{equation*}
with optimal value denoted by $R_{\mathrm{tot}}^{\mathrm{MI}}$.

The budget constraint limits the total deployment cost. The selection-induced bounds enforce $w_i=0$ when $z_i=0$ and $w_i\in[w_{\min},w_{\max}]$ when $z_i=1$, while $\vone^\top\vz=k$ requires exactly $k$ added edges. Since these bounds link $\vz$ and $\vw$, the augmentation Laplacian $\mLa(\vw)$ need not depend explicitly on $\vz$. Setting $w_{\min}=w_{\max}$ and using uniform costs recovers the effective-resistance link-addition problem as a special case, which is \textsc{NP}-hard~\cite{Kooij2023}. 

We relax the binary variables from $\vz\in\{0,1\}^p$ to $\vz\in[0,1]^p$. Direct evaluation of the resulting trace-inverse objective in an iterative solver requires repeated factorizations or linear solves with $\mM(\vw)$, with an $\mathcal O(n^3)$ worst-case cost in dense arithmetic~\cite{golub2013matrix}. To avoid explicit matrix inversion while preserving convexity, we introduce an auxiliary matrix $\mR\succeq0$ and apply the standard Schur-complement lift~\cite{Minimizing_ER}, yielding
\vspace{1em}
$\begin{aligned}
\underset{\mR,\,\vw,\,\vz}{\mathrm{minimize}}\quad &
n\,\operatorname{Tr}(\mR) - n \\[-2pt]
\mathrm{subject\ to}\quad &
\vc^\top \vw \le C, \hspace{120pt} (\mathcal{P}1) \\
& z_i w_{\min} \le w_i \le z_i w_{\max}, \quad i=1,\dots,p, \\
& \vone^\top \vz = k, \quad 0 \le z_i \le 1, \ i=1,\dots,p, \\
& \begin{pmatrix}
\mM(\vw) & \mI \\
\mI & \mR
\end{pmatrix} \succeq 0, \quad \mR \in \mathbb{S}_+^{n}
\end{aligned} $

with optimal value denoted by $R_{\mathrm{tot}}^{\mathrm{SDP}}$.

\begin{proposition}\label{proposition_1}
Assume the base graph is connected, so that $\mM(\vw)\succ 0$ for every feasible $\vw$.
Then 
$R_{\mathrm{tot}}^{\mathrm{SDP}} \le R_{\mathrm{tot}}^{\mathrm{MI}}$,
and every optimum $(\mR^\star,\vw^\star,\vz^\star)$ of $(\mathcal{P}1)$ satifies \(\mR^\star \;=\; \mM^{-1}(\vw^\star).\)
Thus, $(\mathcal{P}1)$ lower-bounds the mixed-integer optimum $R_{\mathrm{tot}}^{\mathrm{MI}}$, and its difference from the objective of any feasible discrete design upper-bounds that design's suboptimality.
\end{proposition}

\begin{proof}[Proof]
Relaxing $\vz\in\{0,1\}^p$ to $\vz\in[0,1]^p$ enlarges the feasible set of $(\mathcal{P}0)$, and hence  $R_{\mathrm{tot}}^{\mathrm{SDP}} \le R_{\mathrm{tot}}^{\mathrm{MI}}$. By the Schur complement \cite{Boyd_Vandenberghe_2004,Semidefinite_programming}, the linear matrix inequality (LMI):
$\bigl[\begin{smallmatrix} \mM(\vw) & \mI\\ \mI & \mR \end{smallmatrix}\bigr]\succeq0$ is equivalent to $\mR\succeq \mM^{-1}(\vw)$.
For fixed feasible $\vw$, any feasible $\mR$ therefore satisfies $\mR-\mM^{-1}(\vw)\succeq0$. If the difference is nonzero, its trace is positive, so replacing $\mR$ by $\mM^{-1}(\vw)$ strictly decreases the objective. Hence, $\mR^\star=\mM^{-1}(\vw^\star)$ at every optimum, recovering the trace-inverse objective as in~\cite{Minimizing_ER}.
\end{proof}

The derivative of total effective resistance and the fixed-topology
KKT condition are established in~\cite[Sec.~4.2]{Minimizing_ER}.
The following proposition adapts them to the heterogeneous costs and
selection-induced bounds of $(\mathcal P1)$.

\begin{proposition}[Cost-normalized marginal-benefit condition]
\label{thm:equal-marginal}
For the relaxed problem $(\mathcal P1)$, define the
cost-normalized marginal benefit of candidate edge $e\in\Ec$ by
\begin{equation}\label{eq:MB}
\Theta_e(\vw)
\coloneqq
\frac{n}{c_e}\|\mM^{-1}(\vw)\va_e\|_2^2
=
-\frac{1}{c_e}
\frac{\partial\Rtot(\vw)}{\partial w_e}.
\end{equation}
Then $\Theta_e(\vw)>0$ for every
$\vw\in\Real_{\ge0}^p$ and $e\in\Ec$.

Let $(\vw^\star,\vz^\star)$ be a KKT point of the reduced
formulation obtained by eliminating $\mR$ via
Proposition~\ref{proposition_1}, and define the set of candidates
whose selection-induced weight bounds are inactive:
\[
S_{\mathrm{int}}^\star
\coloneqq
\{e\in\Ec:
z_e^\star w_{\min}<w_e^\star<z_e^\star w_{\max}\}.
\]
For every $e\in S_{\mathrm{int}}^\star$,
\[
\Theta_e(\vw^\star)=\mu,
\]
where $\mu$ is the multiplier of 
$\vc^\top\vw\le C$. If $S_{\mathrm{int}}^\star\neq\emptyset$,
then $\mu>0$ and the budget is active:
$\vc^\top\vw^\star=C$.
\end{proposition}

\begin{proof}
Using the derivative formula of~\cite[Sec.~4.2]{Minimizing_ER}
to $\mM(\vw)$, and 
$\partial\mM(\vw)/\partial w_e=\va_e\va_e^\top$ gives
\[
\frac{\partial\Rtot(\vw)}{\partial w_e}
=
-n\,\va_e^\top\mM^{-2}(\vw)\va_e
=
-n\,\|\mM^{-1}(\vw)\va_e\|_2^2.
\]
Since $\mM^{-2}(\vw)\succ0$, $\va_e\neq\vzero$, and $c_e>0$,
Eq.~\eqref{eq:MB} follows and $\Theta_e(\vw)>0$.

Let $\mu\ge0$, $\alpha_e\ge0$, and $\gamma_e\ge0$ denote the
multipliers of the budget constraint and the lower and upper
selection-induced weight bounds, respectively. The $w_e$-stationarity
condition is
\[
\frac{\partial\Rtot(\vw^\star)}{\partial w_e}
+\mu c_e-\alpha_e+\gamma_e=0.
\]
For $e\in S_{\mathrm{int}}^\star$, both weight bounds
are strictly slack, so complementary slackness gives
$\alpha_e=\gamma_e=0$. Dividing the resulting stationarity
condition by $c_e>0$ and using~\eqref{eq:MB} yields
$\Theta_e(\vw^\star)=\mu$. If
$S_{\mathrm{int}}^\star\neq\emptyset$, positivity of
$\Theta_e(\vw^\star)$ gives $\mu>0$, and budget complementary
slackness then gives $\vc^\top\vw^\star=C$.
\end{proof}

To recover an integer-feasible design for $(\mathcal P0)$, we apply a top-$k$ \emph{rounding-and-repair} (RND) procedure to the relaxed solution $(\vw^\star,\vz^\star)$ of $(\mathcal P1)$. We select the $k$ largest entries of $\vz^\star$, breaking ties by smaller unit costs $c_i$, to obtain $\widehat{\vz}\in\{0,1\}^p$ with $\vone^\top\widehat{\vz}=k$. If $\vc^\top(w_{\min}\widehat{\vz})>C$, no feasible weights exist on this support, and no rounded upper bound is obtained for the given $(C,k)$ pair. Otherwise, we clip the relaxed weights to the box defined by $\widehat{\vz}$:
\[
\widetilde w_i
=
\min\{\max(w_i^\star,w_{\min}\widehat z_i),
       w_{\max}\widehat z_i\}.
\]
If $\vc^\top\widetilde{\vw}\le C$, we set $\widehat{\vw}=\widetilde{\vw}$. Otherwise, we set
\[
\widehat{\vw} = w_{\min}\widehat{\vz} + \alpha\bigl(\widetilde{\vw}-w_{\min}\widehat{\vz}\bigr),
\qquad \alpha = \frac{C-\vc^\top(w_{\min}\widehat{\vz})}{\vc^\top\widetilde{\vw}-\vc^\top(w_{\min}\widehat{\vz})}.
\]
Under the preceding feasibility condition, $\alpha\in[0,1]$ and $\vc^\top\widehat{\vw}=C$. Hence, $(\widehat{\vw},\widehat{\vz})$ is feasible for $(\mathcal P0)$ and provides the upper bound used in the gap certificate below.

\begin{remark}
For any feasible rounded pair $(\widehat{\vw},\widehat{\vz})$, 
\[R_{\mathrm{tot}}^{\mathrm{SDP}}=\mathrm{val}(\mathcal{P}1)\ \le\ \mathrm{val}(\mathcal{P}0)=R_{\mathrm{tot}}^{\mathrm{MI}}\ \le\ R_{\mathrm{tot}}(\widehat{\vw}) \eqqcolon R_{\mathrm{tot}}^{\mathrm{RND}}.\]
Thus
$R_{\mathrm{tot}}^{\mathrm{RND}}
-R_{\mathrm{tot}}^{\mathrm{SDP}}$
upper-bounds the absolute suboptimality of the rounded design. We
report the normalized certificate
\[
\frac{
R_{\mathrm{tot}}^{\mathrm{RND}}
-R_{\mathrm{tot}}^{\mathrm{SDP}}
}{
R_{\mathrm{tot}}^{\mathrm{SDP}}
}.
\]
The certificate is available only when rounding and repair produce
a feasible pair. If
$\vc^\top(w_{\min}\widehat{\vz})>C$, the selected support cannot
satisfy the lower weight bounds and therefore yields no feasible
upper bound for that $(C,k)$ pair.
\label{remark_1}
\end{remark}

Next, we reformulate $(\mathcal{P}1)$ as a standard cone program and solve it using an operator–splitting method.

\section{Conic optimization via operator splitting}\label{conic_and_adpt_relax}

We solve $(\mathcal{P}1)$ using the standard HSDE--ADMM conic framework of~\cite{o2016conic}; the solver itself is not a contribution of this work. This section specifies how the SDP relaxation is cast into the required conic form. Applying the scaled vectorization $\svec(\cdot)$ to symmetric matrices, with off-diagonal entries scaled by $\sqrt{2}$ to preserve the trace inner product~\cite{wolkowicz2012handbook}, gives:
\begin{equation}
\begin{array}{llrl}
\underset{\vx}{\mathrm{minimize}} & \vq^\top \vx
&  & \hspace{2em} (\mathcal{P}2) \\
\mathrm{subject\ to} & \mF\vx + \vs = \vb\,, & (\vx,\vs) \in \mathbb{R}^{N_{\mathrm{x}}} \times \mathcal{K} &  \\
\end{array}
\notag
\end{equation}
Here, $\vx=[\svec(\mR)^\top\;\vw^\top\;\vz^\top]^\top$, $\vq=[n\svec(\mI)^\top\;\vzero_p^\top\;\vzero_p^\top]^\top$, and $N_{\mathrm{x}}=n(n+1)/2+2p$, so $\vq^\top\vx=n\Tr(\mR)$. The constant $-n$ in $(\mathcal{P}1)$ is omitted here and restored when reporting $R_{\mathrm{tot}}^{\mathrm{SDP}}$.

After $\svec(\cdot)$, the block LMI contributes the semidefinite constraint block $\mF_{\mathrm{SDP}}\vx+\vs_{\mathrm{SDP}}=\vb_{\mathrm{SDP}}$, where $\vs_{\mathrm{SDP}}$ is the vectorized slack associated with a matrix in  $\mathbb{S}^{2n}_{+}$. The matrix $\mF_{\mathrm{SDP}}$ contains the coefficients of $\svec(\mR)$ and $\vw$, with zero columns for $\vz$, and $\vb_{\mathrm{SDP}}$ contains the constant part. The budget, selection-induced weight bounds, and relaxed-selection bounds are represented by nonnegative slack blocks, while $\vone^\top\vz=k$ is represented by a zero-cone block.

Stacking the blocks gives $\mF=[\mF_{\mathrm{SDP}}^\top\;\mF_{\mathrm{cost}}^\top\;\mF_{\mathrm{box}}^\top\;\mF_{\mathrm{card}}^\top]^\top$, $\vb=[\vb_{\mathrm{SDP}}^\top\;C\;\vb_{\mathrm{box}}^\top\;k]^\top$, and $\vs=[\vs_{\mathrm{SDP}}^\top\;\vs_{\mathrm{cost}}^\top\;\vs_{\mathrm{box}}^\top\;\vs_{\mathrm{card}}^\top]^\top\in\mathcal K$, where $N_{\mathrm c}$ is the dimension of $\vb,\vs$, and $\vy$. The product cone is $\mathcal K=\mathbb{S}^{2n}_{+}\times\mathbb{R}_{+}\times(\mathbb{R}_{+}^{p})^4\times\{0\}$. Each candidate edge contributes one sparse rank-one coefficient in the LMI through $\mLa$, while $k$ appears only in the right-hand side of $\vone^\top\vz=k$. Thus the assembled $\mF$ has $\mathcal O(n^2+p)$ nonzeros; forming $\vb$ and $\vq$ costs $\mathcal O(n^2)$ and $\mathcal O(n)$, respectively. The total conic-data assembly cost is therefore $\mathcal O(n^2+p)$, independent of the chosen cardinality $k$.

We solve $(\mathcal{P}2)$ using the HSDE--ADMM scheme of~\cite{o2016conic}. In the embedding, the optimality conditions of $(\mathcal{P}2)$ are represented as $\vv=\mQ\vu$ with $\vu=(\vx,\vy,\tau)\in \mathcal{C}\coloneqq\mathbb{R}^{N_{\mathrm{x}}}\times\mathcal K^*\times\mathbb{R}_+$ and $\vv=(\vr,\vs,\kappa)\in\mathcal{C}^*\coloneqq\{0\}^{N_{\mathrm{x}}}\times\mathcal K\times\mathbb{R}_+$, where
\[
\mQ=\begin{bmatrix}
0 & \mF^\top & \vq\\
-\mF & 0 & \vb\\
-\vq^\top & -\vb^\top & 0
\end{bmatrix}.
\]
The variables $\tau,\kappa\ge0$ distinguish solutions from infeasibility certificates: if $\tau>0,\kappa=0$, then $(\vx/\tau,\vy/\tau,\vs/\tau)$ satisfies the KKT conditions of $(\mathcal{P}2)$; if $\tau=0,\kappa>0$, the embedding yields a primal or dual infeasibility certificate~\cite[Sec.~2.3]{o2016conic}. Applying ADMM to this embedding gives the standard iteration of~\cite[Eq.~(17)]{o2016conic}, summarized in Algorithm~\ref{alg:hsde}.

\begin{algorithm}[t]
\caption{HSDE-ADMM}
\label{alg:hsde}
\KwIn{$\mF,\,\vb,\,\vq,\,\Kcone$ and the induced $\mQ,\Cset$; tolerance $\varepsilon\!>\!0$; max.\ iterations $T$}
\KwOut{ $(\vx^{\star},\vy^{\star})$ \textit{or} infeasibility status}
\textbf{Initialization:} $\vu^{0}\leftarrow[\vzero;\,1],\;\; \vv^{0}\leftarrow[\vzero;\,1]$\;
\For{$t = 0, 1, \ldots, T-1$}{
     $\vomega^{t} \leftarrow \vu^{t} + \vv^{t}$  \Comment{primal sum}\;
     $\tilde\vu^{t+1} = (\mI+\mQ)^{-1}\,\vomega^{t}$  \Comment{linear system}\label{ln:affine}\;
     $\vu^{t+1} \leftarrow \proj_{\Cset}\,\bigl(\tilde\vu^{t+1}-\vv^{t}\bigr)$ \Comment{cone projection}\label{ln:cone}\;
     $\vv^{t+1} \leftarrow \vv^{t}-\tilde\vu^{t+1}+\vu^{t+1}$ \Comment{dual update}\label{ln:dual}\;
     $r_t\gets \dfrac{\| \mtrx{Q}\vu^{t+1}-\vv^{t+1}\|_2}{\max\{1,\,\|\vu^{t+1}\|_2,\,\|\vv^{t+1}\|_2\}} $ \Comment{relative residual}\label{ln:res}\;
    \If{$r_t < \varepsilon$}{
        \textbf{break}
    }
}
$\tau \leftarrow \vu^{t+1}_{N_{\mathrm{x}}+N_{\mathrm{c}}+1},\;\; \kappa \leftarrow \vv^{t+1}_{N_{\mathrm{x}}+N_{\mathrm{c}}+1}$\;
\If{$\tau > \varepsilon$ \textbf{and} $\kappa < \varepsilon$}
    {\textbf{return} \, $\{\vx^{\star},\, \vy^{\star}\} = \left\{\dfrac{\vu^{t+1}_{1:N_{\mathrm{x}}}}{\tau},\, \dfrac{\vu^{t+1}_{N_{\mathrm{x}}+1:N_{\mathrm{x}}+N_{\mathrm{c}}}}{\tau}  \right\}$}
\ElseIf{$\tau < \varepsilon$ \textbf{and} $\kappa > \varepsilon$}{
     \Return \textsc{infeasible}\;
}
\Else{
     \Return \textsc{unknown}\;
}
\end{algorithm}

The subspace-projection step in line~\ref{ln:affine} requires solving a fixed linear system. In a direct implementation, an equivalent sparse quasi-definite system is factorized once using a permuted $LDL^\top$ factorization, and the factors are reused through forward and backward substitutions at subsequent iterations \cite[Sec.~4.1]{o2016conic}. The factorization and solve costs depend on the fill-in produced by the sparsity pattern and ordering.

The cone projection in line~\ref{ln:cone} decomposes across the cone blocks. Its dominant component is projection onto $\mathbb S_+^{2n}$, which requires an eigendecomposition of the $2n\times2n$ block and has worst-case cost $\mathcal O(n^3)$. The remaining orthant and free-coordinate operations cost $\mathcal O(p)$, while the update in line~\ref{ln:dual} costs $\mathcal O(N_{\mathrm{x}}+N_{\mathrm c}) =\mathcal O(n^2+p)$.

The HSDE is homogeneous: scaling a solution $(\vu,\vv)$ by $\psi>0$ gives the same recovered primal--dual point or infeasibility certificate~\cite[Sec.~2.3]{o2016conic}. The raw embedding residual $\|\mQ\vu^t-\vv^t\|_2$ therefore scales with the iterate magnitude and is not a suitable absolute stopping measure. We instead use the normalized residual in line~\ref{ln:res} and stop when $r_t<\varepsilon$. The unit floor prevents division by a small normalizer during early iterations. The residual $r_t$ measures embedding feasibility rather than the recovered KKT accuracy directly. Section~\ref{sec:results} therefore also reports the primal, dual, and duality-gap residuals for instances N$1$ to N$4$ (Fig.~\ref{fig:solver_convergence} and Table~\ref{tab:solver_convergence}).

\begin{table}[!h]
  \centering
  \footnotesize
  \renewcommand{\arraystretch}{1.08}
  \caption{Problem-size settings of the experimental instances N$1$ to N$4$ (Section~\ref{sec:experimental-setup}).}
  \label{tab:problem-sizes}
  \begin{tabular*}{\columnwidth}{@{\extracolsep{\fill}}lrrrr@{}}
    
    \toprule
    & \textbf{small} & \textbf{medium} & \textbf{large} & \textbf{xlarge} \\
    \midrule
    Instance label & N$1$ & N$2$ & N$3$ & N$4$ \\
    Nodes $n$ & $50$ & $100$ & $500$ & $1{,}000$ \\
    Base edges $|\Eb|$ & $150$ & $300$ & $1{,}500$ & $3{,}000$ \\
    Candidate edges $p=|\Ec|$ & $83$ & $178$ & $926$ & $1{,}850$ \\
    \addlinespace[2pt]
    Primal variables $N_{\mathrm{x}}$ & $1{,}441$ & $5{,}406$ & $127{,}102$ & $504{,}200$ \\
    Inequality constraints & $333$ & $713$ & $3{,}705$ & $7{,}401$ \\
    PSD block size $2n$ & $100$ & $200$ & $1{,}000$ & $2{,}000$ \\
    $\svec$ dim.\ of PSD block & $5{,}050$ & $20{,}100$ & $500{,}500$ & $2{,}001{,}000$ \\
    Total constraint rows & $5{,}384$ & $20{,}814$ & $504{,}206$ & $2{,}008{,}402$ \\
    \addlinespace[2pt]
    Nonzeros in $F$ & $2{,}188$ & $7{,}008$ & $135{,}436$ & $520{,}850$ \\
    
    \bottomrule
  \end{tabular*}
\end{table}

\textbf{Scaling.} Table~\ref{tab:problem-sizes} shows the dimensions of the four conic formulations. The direct HSDE implementation reuses a fixed
sparse factorization, whereas an interior-point method typically forms and factors an iteration-dependent KKT or Schur-complement system at each Newton step, making memory and factorization costs restrictive for the largest instances~\cite{nesterov_interior}. For N$4$, the dominant HSDE operation is the eigendecomposition of the $2{,}000\times2{,}000$ PSD block. Section~\ref{sec:results} reports the observed iteration counts, residuals, and runtimes.

\section{Budgeted Greedy Augmentation via Rank-One Updates}\label{budget_greedy}

{\setlength{\textfloatsep}{4pt plus 1pt minus 1pt}
\begin{algorithm}[!t]
\caption{Greedy Budgeted Weight Optimization}
\label{alg:greedy}
\KwIn{Shifted Laplacian: $\mM'_0$; candidate edge set $\Ec$, associated edge costs $\{c_e\}$; target cardinality $k$; weight bounds $[w_{\min}, w_{\max}]$; budget $C$.}
\KwOut{Augmented graph $\mM'$ with exactly $k$ added edges whenever $(C,k)$ is feasible; otherwise no exact-$k$ feasible design is reported.}
\textbf{Initialization:} $\Eused^{(0)} \leftarrow \emptyset$, $\Erem^{(0)} \leftarrow \Ec$, $C_{\mathrm{rem}}^{(0)} \leftarrow C$ \, form $(\mM'_0)^{-1}$ and $\bigl((\mM'_0)^{-1}\bigr)^{2}$\;
\If{$|\Ec| < k$ \textbf{ or } $C < w_{\min}\sum_{i=1}^{k} c_{(i)}$}{
    $\mathrm{return}$ \Comment{$(C,k)$ infeasible; no exact-$k$ design}\;
}
\For{$l = 1,\dots, k$}{
    $k_{\mathrm{rem}}^{(l-1)} \leftarrow k - |\Eused^{(l-1)}|$\;
    $r \leftarrow k_{\mathrm{rem}}^{(l-1)}$\;
    $s \leftarrow \dfrac{C_{\mathrm{rem}}^{(l-1)}}{r}$ \Comment{Uniform spend}\;
    $\Delta R_{\mathrm{tot}}^{\mathrm{best}} \leftarrow 0$ \Comment{Best marginal gain}\;
    $ e^{\star}  \leftarrow \emptyset$ \Comment{Greedy edge} \; $w_{e^{\star}}^{\star} \leftarrow 0$ \Comment{Greedy edge weight}\;
    \For{$e = \{v_i,v_j\} \in \Erem^{(l-1)}$}{
        $R_e \leftarrow w_{\min}\sum_{i=1}^{r-1} c_{(i)}$ \Comment{Reserve; $c_{(i)}$ $i$-th cheapest in $\Erem\!\setminus\!\{e\}$}\;
        \If{$C_{\mathrm{rem}}^{(l-1)} - c_e\,w_{\min} < R_e$}{
            $\mathrm{continue}$ \Comment{$\mathrm{Inadmissible}$}\;
        }
        $w_e \leftarrow \mathrm{Eq}.~\eqref{weight_greedy}$ \Comment{Reserve-capped trial weight}\;
        $
            r_e  \leftarrow   \left[(\mM'_{l-1})^{-1}\right]_{(v_i,v_i)} + \left[(\mM'_{l-1})^{-1}\right]_{(v_j,v_j)}
        $
        $
            \hspace{5em}-2\,\left[(\mM'_{l-1})^{-1}\right]_{(v_i,v_j)}\;
        $\;
        $
            b_e^{2}\leftarrow
            \left[(\mM'_{l-1})^{-2}\right]_{(v_i,v_i)}
        $
        \[
            \hspace{.3em}+ \left[(\mM'_{l-1})^{-2}\right]_{(v_j,v_j)}- 2\,\left[(\mM'_{l-1})^{-2}\right]_{(v_i,v_j)};
        \]\DontPrintSemicolon\;
        \PrintSemicolon
        \vspace{-1em}
        $\Delta R_{\mathrm{tot}}(e) \leftarrow \mathrm{Eq}.~\eqref{marginal_gain_weighted}$\;
        \If{$\Delta R_{\mathrm{tot}}(e) > \Delta R_{\mathrm{tot}}^{\mathrm{best}}$}{
               $\Delta R_{\mathrm{tot}}^{\mathrm{best}} \leftarrow \Delta R_{\mathrm{tot}}(e)$\;
               $e^{\star}  \leftarrow e$\;
               $w_{e^{\star}}^{\star} \leftarrow w_e$\;
        }
    }
    $C_{\mathrm{rem}}^{(l)} \leftarrow C_{\mathrm{rem}}^{(l-1)} - c_{e^{\star}}\,w_{e^{\star}}^{\star}$\;
    $\Erem^{(l)} \leftarrow \Erem^{(l-1)} \backslash \{e^{\star}\}$\;
    $\Eused^{(l)} \leftarrow \Eused^{(l-1)} \cup \{e^{\star}\}$\;
    \hspace{-5em}\Comment{Augment network with selected edge}\;
    Update $(\mM'_l)^{-1} \,\mathrm{using}\, \mathrm{Eq}.~\eqref{eq:sherman_morrison}$\;
    {Update $\left((\mM'_l)^{-1}\right)^{2} \,\mathrm{using}\, \mathrm{Eq}.~\eqref{eq:Deter-update}$}\;
}
\end{algorithm}}

Although $(\mathcal P1)$ provides a lower bound and supports feasible recovery by rounding, solving it remains costly at large scale. We therefore use a budget-feasible greedy heuristic. At each iteration, the method selects the admissible candidate with the largest marginal reduction in total effective resistance. A reserve retains enough budget to assign the minimum weight $w_{\min}$ to all remaining additions. Proposition~\ref{prop:exactk_feasibility} therefore guarantees exactly $k$ added edges whenever $(C,k)$ is feasible. Sherman-Morrison updates replace repeated $\mathcal O(n^3)$ factorizations by $\mathcal O(n^2)$ rank-one updates.

Let $\mM'_0\coloneqq\mLb+\vone\vone^\top/n$, and let $\mM'_l$ denote the shifted Laplacian after $l$ additions. Adding $e=\{v_i,v_j\}$ with incidence vector $\va_e=\ve_{v_i}-\ve_{v_j}$ and weight $w_e$ gives $\mM'_l=\mM'_{l-1}+w_e\va_e\va_e^\top$. By the Sherman-Morrison identity~\cite{ShermanMorrison1950},
\begin{equation}\label{eq:sherman_morrison}
\mM'^{-1}
= \mM^{-1}
- \frac{w_e\,\mM^{-1}\va_e\va_e^\top \mM^{-1}}{1+w_e\,\va_e^\top \mM^{-1}\va_e}.
\end{equation}
At the current shifted Laplacian, define $r_e=\va_e^\top\mM^{-1}\va_e$ as the effective resistance between the endpoints and $b_e^2=\|\mM^{-1}\va_e\|_2^2$ as their squared biharmonic distance. The marginal reduction in total effective resistance is
\begin{equation}\label{marginal_gain_weighted}
    \Delta R_{\mathrm{tot}}(e) = \frac{n\,w_e\,b_e^{2}}{1+w_e\,r_e}.
\end{equation}
A key design choice is how to assign a trial weight $w_e$ to each candidate during the greedy scan. Rather than solving a weight-optimization subproblem at every step, we use a simple cost-aware provisional spend $s \coloneqq C_{\mathrm{rem}}/k_{\mathrm{rem}}$ over the remaining additions, then cap it so the reserve $R_e$ needed to complete the other $k_{\mathrm{rem}}-1$ additions at the floor $w_{\min}$ is left untouched:
\begin{equation}\label{weight_greedy}
w_e \coloneqq \min\!\left\{\max\left\{w_{\min},\min\left\{w_{\max},\frac{s}{c_e}\right\}\right\},\frac{C_{\mathrm{rem}}-R_e}{c_e}\right\},
\end{equation}
where $R_e \coloneqq w_{\min}\sum_{i=1}^{k_{\mathrm{rem}}-1} c_{(i)}$, with $c_{(1)}\le c_{(2)}\le\cdots$ the ascending costs of $\Erem\setminus\{e\}$ (and $R_e\coloneqq 0$ when $k_{\mathrm{rem}}=1$). For any admissible candidate, i.e.\ $c_e w_{\min}\le C_{\mathrm{rem}}-R_e$, the cap never lowers $w_e$ below $w_{\min}$, so $w_e\in[w_{\min},w_{\max}]$.

\begin{proposition}[Exact-$k$ feasibility of reserve-capped greedy]
\label{prop:exactk_feasibility}
Let $0<c_e<\infty$ for all $e\in\Ec$ and $0<w_{\min}\le w_{\max}$, and order the candidate costs as $c_{(1)}\le\cdots\le c_{(p)}$. Then $(\mathcal P0)$ is feasible for $(C,k)$ if and only if $k\le p$ and $C\ge w_{\min}\sum_{i=1}^{k}c_{(i)}$. Under this condition, Algorithm~\ref{alg:greedy} admits an admissible candidate at every iteration and returns a feasible exact-$k$ solution of $(\mathcal P0)$.
\end{proposition}

\begin{proof}
Let $\Phi_m(S)$ denote the sum of the $m$ smallest costs in
a finite set $S$, with $\Phi_0(S)=0$. The conditions are necessary because
any feasible exact-$k$ design selects $k$ edges. Each selected edge has
weight at least $w_{\min}$. Therefore, the total cost is at least
$w_{\min}\Phi_k(\Ec)=w_{\min}\sum_{i=1}^{k}c_{(i)}$.

For sufficiency, we prove by induction that
$|\Erem^{(l-1)}|\ge r$ and
$C_{\mathrm{rem}}^{(l-1)}\ge
w_{\min}\Phi_r(\Erem^{(l-1)})$, where
$r=k_{\mathrm{rem}}^{(l-1)}$. It holds at $l=1$ by the
feasibility condition. At iteration $l$, define the reserve
$R_e=w_{\min}\Phi_{r-1}(\Erem^{(l-1)}\setminus\{e\})$.
Under the invariant, any edge among the $r$ cheapest residual
edges is admissible, since
$c_e w_{\min}+R_e=w_{\min}\Phi_r(\Erem^{(l-1)})
\le C_{\mathrm{rem}}^{(l-1)}$. Hence the admissible set is
nonempty. For the selected admissible edge $e^\star$, the
reserve cap in~\eqref{weight_greedy} gives
$C_{\mathrm{rem}}^{(l)}
=C_{\mathrm{rem}}^{(l-1)}-c_{e^\star}w_{e^\star}
\ge R_{e^\star}
=w_{\min}\Phi_{r-1}(\Erem^{(l)})$.
Moreover, $|\Erem^{(l)}|\ge r-1$, so the invariant holds for the
next iteration. Induction gives $k$ selected edges, each with weight
in $[w_{\min},w_{\max}]$, and total cost at most $C$.
\end{proof}

The method first forms
$\mA_0\coloneqq(\mM'_0)^{-1}$ and $\mA_0^2$, which costs
$\mathcal O(n^3)$ in dense arithmetic. The entries of $\mA$ give
$r_e$ in constant time. Following the biharmonic-distance caching
idea of~\cite{zhou2025efficient}, the squared distance is also
obtained in constant time from $\mA^2$:
\begin{equation}
b_e^2
=
[\mA^2]_{v_i,v_i}
+
[\mA^2]_{v_j,v_j}
-
2[\mA^2]_{v_i,v_j}.
\end{equation}

For the selected edge, let
$\mathbf{g}_e\coloneqq\mA\va_e$ and
$\chi_e\coloneqq w_e/(1+w_e r_e)$. The inverse and squared-inverse
caches are updated as
\begin{equation}
\mA'
=
\mA-\chi_e\mathbf{g}_e\mathbf{g}_e^\top
\end{equation}
and
\begin{equation}\label{eq:Deter-update}
\begin{aligned}
(\mA')^2
={}&
\mA^2
-\chi_e\bigl(
\mA\mathbf{g}_e\mathbf{g}_e^\top
+
\mathbf{g}_e\mathbf{g}_e^\top\mA
\bigr)+
\chi_e^2\|\mathbf{g}_e\|_2^2
\mathbf{g}_e\mathbf{g}_e^\top .
\end{aligned}
\end{equation}
Both cache updates cost $\mathcal O(n^2)$.

After the one-time ordering of the candidate costs, the reserve, effective resistance, and biharmonic distance of each candidate are evaluated in $\mathcal O(1)$ time. Each iteration therefore requires an $\mathcal O(p)$ candidate scan and $\mathcal O(n^2)$ cache updates. The total complexity is $\mathcal O\left(n^3+k(p+n^2)\right).$ Storing $\mA$ and $\mA^2$ requires $\mathcal O(n^2)$ memory. The full procedure is summarized in Algorithm~\ref{alg:greedy}. The analysis below does not claim a global approximation ratio. Instead, Theorem~\ref{thm:policy_bound_alg2} characterizes the decay of the Bellman-optimal value-to-go along the realized greedy trajectory. Theorem~\ref{thm:explicit_rho_alg2} then gives a conservative spectral lower bound on the associated local policy ratio.

We compare the greedy policy with a sequentially optimal policy over the reachable states induced by the same reserve-capped weight rule. Each state records the current augmented matrix, remaining candidate set, selected edges, and residual budget. One-step residual-gain comparisons of this type underlie classical greedy and weak-submodular analyses~\cite{NemhauserWolseyFisher1978,DasKempe2018}, with related policy comparisons used in adaptive
settings~\cite{GolovinKrause2011}. We therefore define Bellman-optimal and greedy value-to-go functions and a state-dependent local policy ratio.
\begin{definition}[Bellman-optimal and greedy value-to-go]
\label{def:value_functions_alg2}
For each iteration \(l\in\{1,\dots,k\}\), define the reachable state
\(
\mathcal X_l
\coloneqq
\bigl(
\mM'_{l-1},\,
\Erem^{(l-1)},\,
\Eused^{(l-1)},\,
C_{\mathrm{rem}}^{(l-1)}
\bigr),
\,\) and \(
k_{\mathrm{rem}}^{(l-1)}
\coloneqq
k-\lvert \Eused^{(l-1)}\rvert.
\)
With $r\coloneqq k_{\mathrm{rem}}^{(l-1)}$, let $R_e^{(l)}$ be the reserve of Eq.~\eqref{weight_greedy} evaluated at state $\mathcal X_l$, that is, the least floor-cost $w_{\min}\sum_{i=1}^{r-1}c_{(i)}$ of the $r-1$ cheapest candidates of $\Erem^{(l-1)}\setminus\{e\}$ (with $R_e^{(l)}\coloneqq 0$ when $r=1$). The feasible action set is
\(
\mathcal F_l
\coloneqq
\Bigl\{
e\in \Erem^{(l-1)}
:\;
c_e w_{\min}\le C_{\mathrm{rem}}^{(l-1)}-R_e^{(l)}
\Bigr\},
\)
and the transition function \(f(\mathcal X_l,e)\) produces the next state \(\mathcal X_{l+1}\) by committing \(e\): the shifted Laplacian advances to \(\mM'_l\) (its inverse maintained via the rank-one update~\eqref{eq:sherman_morrison}), the budget is decremented as \(C_{\mathrm{rem}}^{(l)}=C_{\mathrm{rem}}^{(l-1)}-c_e w_e^{(l)}\), and the edge sets are updated to \(\Erem^{(l)}=\Erem^{(l-1)}\setminus\{e\}\) and \(\Eused^{(l)}=\Eused^{(l-1)}\cup\{e\}\). For each candidate \(e\in\mathcal F_l\), the one-step gain at state \(\mathcal X_l\) is \(\Delta R_{\mathrm{tot}}^{(l)}(e) \coloneqq n\,w_e^{(l)}\bigl(b_e^{2}\bigr)^{(l)}/\bigl(1+w_e^{(l)}\,r_e^{(l)}\bigr)\). The weight \(w_e^{(l)}\) entering this gain is the reserve-capped trial weight of Eq.~\eqref{weight_greedy}; for every \(e\in\mathcal F_l\), admissibility gives \(w_e^{(l)}\in[w_{\min},w_{\max}]\) and \(c_e w_e^{(l)}\le C_{\mathrm{rem}}^{(l-1)}-R_e^{(l)}\), so committing \(e\) leaves the reserve \(R_e^{(l)}\) intact and the residual instance stays exact-\(k\) feasible.
The Bellman-optimal value-to-go from \(\mathcal X_l\) is defined recursively:
\begin{equation*}
V_l(\mathcal X_l)=
\begin{cases}
0,
& \hspace{-10em}
\text{if }\mathcal F_l=\varnothing
\text{ or }
k_{\mathrm{rem}}^{(l-1)}=0,\\[0.6em]
\displaystyle
\max_{e\in\mathcal F_l}
\Bigl\{
\Delta R_{\mathrm{tot}}^{(l)}(e)
+
V_{l+1}\bigl(f(\mathcal X_l,e)\bigr)
\Bigr\},
& \hspace{-.6em}
\text{otherwise.}
\end{cases}
\end{equation*}
Whenever \(\mathcal F_l\neq\varnothing\) and \(k_{\mathrm{rem}}^{(l-1)}\ge 1\), the greedy action at iteration \(l\) selects the edge \(e_l^\star = \arg\max_{e\in\mathcal F_l}\Delta R_{\mathrm{tot}}^{(l)}(e)\), with ties broken by enumeration order. The value accumulated by Algorithm~\ref{alg:greedy} from state \(\mathcal X_l\) is
\begin{equation*}
G_l(\mathcal X_l)=
\begin{cases}
0,
& \hspace{-10em}
\text{if }\mathcal F_l=\varnothing
\text{ or }
k_{\mathrm{rem}}^{(l-1)}=0,\\[0.6em]
\displaystyle
\Delta R_{\mathrm{tot}}^{(l)}(e_l^\star)
+
G_{l+1}\bigl(f(\mathcal X_l,e_l^\star)\bigr),
&
\text{otherwise.}
\end{cases}
\end{equation*}
At the terminal stage we set \(V_{k+1}(\mathcal X_{k+1})\coloneqq 0\) and \(G_{k+1}(\mathcal X_{k+1})\coloneqq 0\), reflecting that no further iterations remain after~\(k\).
Thus, $V_l(\mathcal X_l)$ is the maximum cumulative gain obtainable under the reserve-capped trial-weight rule, whereas $G_l(\mathcal X_l)$ is the gain produced by the myopic action sequence of Algorithm~\ref{alg:greedy}.
\end{definition}

\begin{definition}[Local policy ratio]
\label{def:local_policy_ratio_alg2}
For each iteration \(l\in\{1,\dots,k\}\) with
\(\mathcal F_l\neq\varnothing\) and \(k_{\mathrm{rem}}^{(l-1)}\ge 1\),
let \(\rho_l\in[0,1]\) be the largest scalar for which
\begin{equation}
\max_{e\in\mathcal F_l}
\Delta R_{\mathrm{tot}}^{(l)}(e)
\;\ge\;
\frac{\rho_l}{k_{\mathrm{rem}}^{(l-1)}}
\,V_l(\mathcal X_l)
\nonumber
\end{equation}
holds. We call \(\rho_l\) the \emph{local policy ratio} at state \(\mathcal X_l\).
\end{definition}
Since $\mathcal F_l$ is finite and nonempty and every admissible gain is positive, $V_l(\mathcal X_l)>0$ and $\rho_l$ is well-defined. The ratio measures the fraction of the Bellman-optimal value-to-go captured by the best admissible one-step action after normalization by the number of remaining additions.

\begin{theorem}[One-step greedy progress, Bellman residual decay, and greedy-value bound]
\label{thm:policy_bound_alg2}
Let \(\rho_l\) be the local policy ratio of Definition~\ref{def:local_policy_ratio_alg2}, and suppose \((C,k)\) is feasible in the sense of Proposition~\ref{prop:exactk_feasibility}. Then every reachable state \(\mathcal X_l\) with \(l\in\{1,\dots,k\}\) has \(\mathcal F_l\neq\varnothing\) and \(k_{\mathrm{rem}}^{(l-1)}\ge 1\), and the greedy edge \(e_l^\star\) of Definition~\ref{def:value_functions_alg2} satisfies:
\begin{enumerate}[label=(\roman*),ref=\roman*]
    \item\label{cond:progress_alg2} $\Delta R_{\mathrm{tot}}^{(l)}(e_l^\star) \ge \dfrac{\rho_l}{k_{\mathrm{rem}}^{(l-1)}}\,V_l(\mathcal X_l),$
    \item\label{cond:contraction_alg2} $V_{l+1}\!\bigl(f(\mathcal X_l,e_l^\star)\bigr) \le \left(1-\dfrac{\rho_l}{k_{\mathrm{rem}}^{(l-1)}}\right)V_l(\mathcal X_l),$
    \item\label{cond:greedy_alg2} $0 \le G_l(\mathcal X_l) \le V_l(\mathcal X_l).$
\end{enumerate}
Moreover, along the realized greedy trajectory
\(\mathcal X_{t+1}=f(\mathcal X_t,e_t^\star)\), one has, for every
\(m\in\{l,\dots,k\}\),
\begin{equation}
\label{eq:product_decay}
V_{m+1}(\mathcal X_{m+1})
\le
\prod_{t=l}^{m}
\left(
1-\frac{\rho_t}{k_{\mathrm{rem}}^{(t-1)}}
\right)
V_l(\mathcal X_l).
\end{equation}
If, in addition, $\rho_t\ge\rho_{\min}>0$ for all $t$, then Proposition~\ref{prop:exactk_feasibility} gives $k_{\mathrm{rem}}^{(t-1)}=k-(t-1)$, and
\begin{equation}
\label{eq:product_decay_uniform}
V_{m+1}(\mathcal X_{m+1})
\le
e^{-\rho_{\min}\bigl(H_{k-l+1}-H_{k-m}\bigr)}\,V_l(\mathcal X_l),
\end{equation}
with \(H_q\coloneqq \sum_{j=1}^q \frac1j\) and \(H_0\coloneqq 0\).
\end{theorem}

\begin{proof}
The proof is given in Appendix~\ref{app:policy_bound}.
\end{proof}

\begin{remark}
Theorem~\ref{thm:policy_bound_alg2} controls the decay of the
Bellman-optimal residual along the greedy trajectory, but it does not by itself
imply a direct global lower bound of the form
\(G_1\ge (1-e^{-\rho_{\min}})V_1\). Such a guarantee requires additional structure
linking the realized successor state of the greedy action to the residual value
function, as occurs in classical set-function analyses and in stronger
policy-based formulations.
\end{remark}

Equation~\eqref{eq:product_decay} retains the state-dependent ratios $\rho_t$ and is therefore sharper than the uniform exponential envelope in~\eqref{eq:product_decay_uniform}. To make the latter computable, the next theorem lower-bounds every $\rho_l$ using spectral envelopes of the augmented matrices.
\begin{theorem}[Uniform local policy ratio from spectral envelopes]
\label{thm:explicit_rho_alg2}
Assume that every matrix \(\mM'_l\) is symmetric positive definite, and define the matrix
\(
\mM'_{\max}\coloneqq
\mM'_0 + w_{\max}\sum_{e\in \Ec}\va_e\va_e^\top
\)
together with the spectral envelopes
\(
\lambda_{\min}^{\mathrm{lower}}
\coloneqq \lambda_{\min}(\mM'_0)>0
\)
and
\(
\lambda_{\max}^{\mathrm{upper}}
\coloneqq
\lambda_{\max}(\mM'_{\max}).
\)
Then the local policy ratio admits the uniform lower bound
\begin{equation}
\label{eq:explicit_rho_alg2}
\rho_l
\;\ge\;
\frac{w_{\min}}{w_{\max}}
\cdot
\frac{\bigl(\lambda_{\min}^{\mathrm{lower}}\bigr)^2}{\bigl(\lambda_{\max}^{\mathrm{upper}}\bigr)^2}
\cdot
\frac{1}{1+\dfrac{2w_{\max}}{\lambda_{\min}^{\mathrm{lower}}}}
\;=:\;
\rho_{\min},
\end{equation}
valid whenever the local policy ratio $\rho_l$ is defined.
\end{theorem}

\begin{proof}
The proof is given in Appendix~\ref{app:explicit_rho}.
\end{proof}

Both bounds are illustrated  in Section~\ref{sec:results} (Fig.~\ref{fig:greedy_decay} and
Fig.~\ref{fig:spectral_envelopes}).

\begin{figure}[!t]
    \centering
    \includegraphics[width=\linewidth]{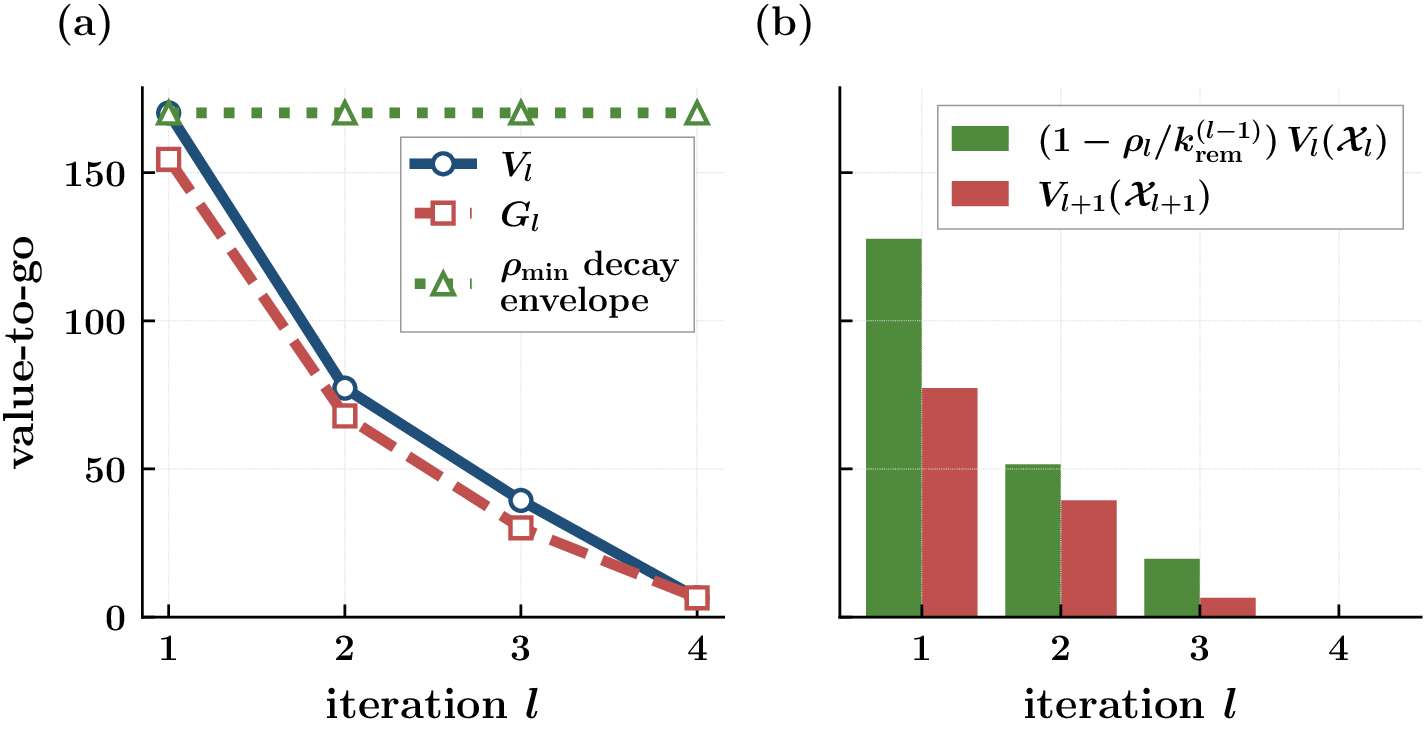}
    \caption{Greedy value bound and Bellman-residual decay
    (Theorem~\ref{thm:policy_bound_alg2}) on a stress-case instance.
    (a)~optimal value-to-go \(V_l\), greedy value \(G_l\), and the
    \(\rho_{\min}\) decay envelope
    \(V_1\prod_{t=1}^{l-1}(1-\rho_{\min}/k_{\mathrm{rem}}^{(t-1)})\).
    (b)~realized next value \(V_{l+1}\) versus the contraction ceiling
    \((1-\rho_l/k_{\mathrm{rem}}^{(l-1)})\,V_l\).}
    \label{fig:greedy_decay}
\end{figure}
\begin{figure}[!t]
    \centering
    \includegraphics[width=\linewidth]{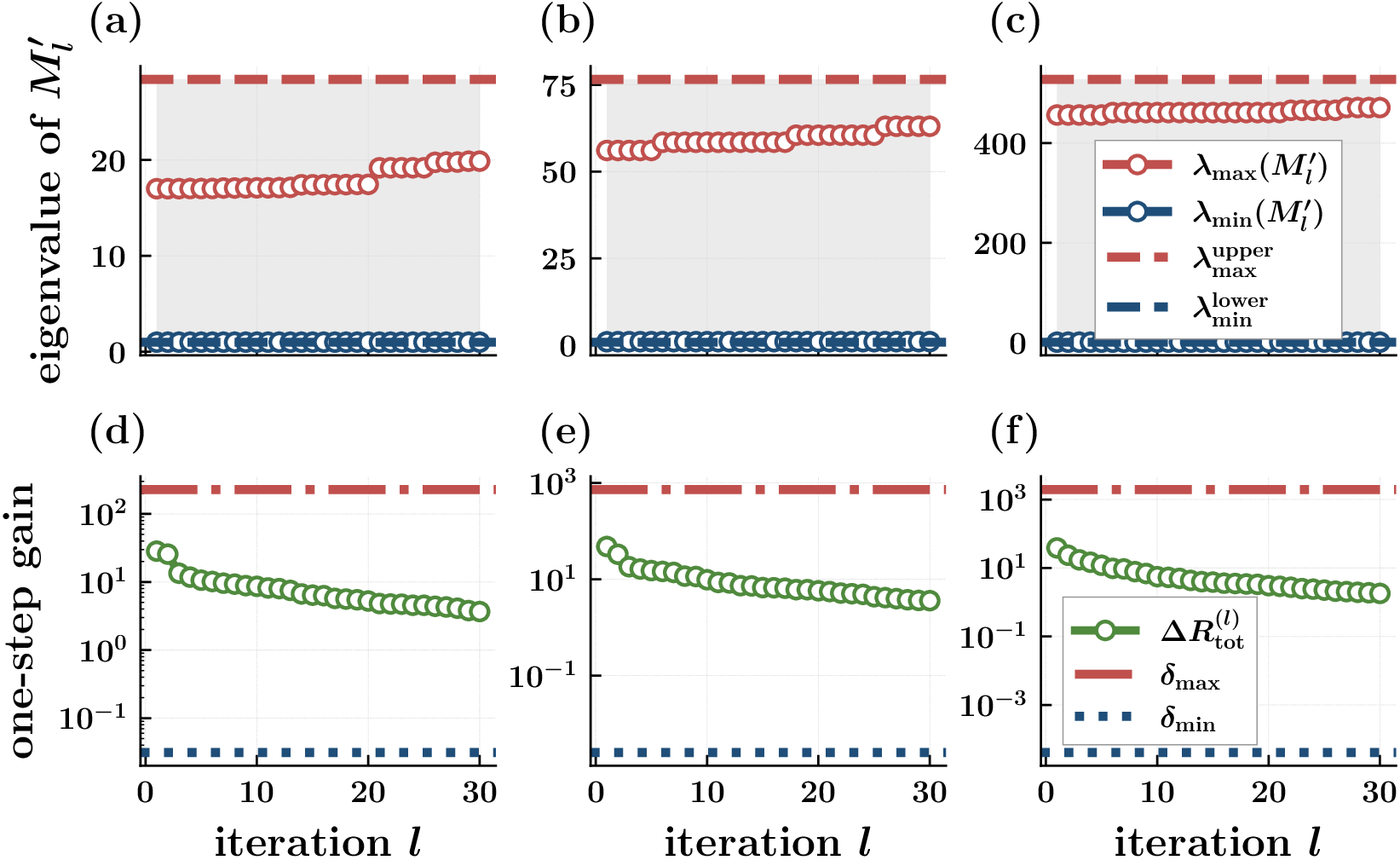}
    \caption{Spectral envelopes of Theorem~\ref{thm:explicit_rho_alg2} on the
    \(n=50\) network (columns: \(\sigma\in\{0.5,1.0,1.5\}\)). Top row (a) to (c):
    eigenvalues \(\lambda_{\max}(\mM'_l)\), \(\lambda_{\min}(\mM'_l)\) and their
    envelopes. Bottom row (d) to (f): per-step gain \(\Delta\Rtot^{(l)}\) and its
    \([\delta_{\min},\delta_{\max}]\) band.}
    \label{fig:spectral_envelopes}
\end{figure}
\section{Experimental Setup}
\label{sec:experimental-setup}

We evaluate the algorithms on four synthetic instances of increasing scale, N$1$ to N$4$, whose node counts $n$ and base-edge counts $|\Eb|$ are listed in Table~\ref{tab:problem-sizes}. All four instances share an average degree of 6, which isolates the effect of graph size from that of density. Node positions are drawn uniformly in the unit square $[0,1]^2$, and the base edges are sampled uniformly from the $\binom{n}{2}$ node pairs. Connectivity is enforced by adding the minimum number of bridging edges. Base edge weights follow a lognormal distribution $w\sim\mathrm{Lognormal}(\mu,\sigma)$, where $\mu$ and $\sigma$ are the mean and standard deviation of $\ln w$. This provides a simple positive, right-skewed weight model, consistent with heterogeneous weights observed in real networks~\cite{barrat2004architecture}. We hold the geometry, candidate set, and cost model fixed while sweeping the base-weight dispersion through three $(\mu,\sigma)$ settings. The dispersion is measured by the coefficient of variation
\(
\mathrm{CV}=\sqrt{e^{\sigma^2}-1},
\)
which depends only on $\sigma$ for a lognormal distribution. The three regimes span increasing dispersion. The baseline $(\mu,\sigma)=(0,0.5)$ is nearly homogeneous, with median $1$ and $\mathrm{CV}\approx0.53$. The intermediate regime $(0,1.0)$ has broader base-weight heterogeneity and $\mathrm{CV}\approx1.31$. The heavy-tailed stress test $(0.5,1.5)$ has median $\approx1.65$, mean $\approx5.08$, and $\mathrm{CV}\approx2.91$. This heavy tail can increase the Laplacian condition number and make the SDP solve more demanding.

The candidate set $\Ec$ is constructed with a clustered $k$-nearest neighbor ($k$-NN) rule. A $k$-means step partitions the nodes into $K = 3$ clusters by their coordinates. Each node is then linked to its $k_{\mathrm{in}} = 3$ nearest neighbors inside its cluster, and a single backbone edge joins the closest node pair between every two clusters. Any candidate already present in $\Eb$ is removed. This yields a structured candidate pool with dense intra-cluster and sparse inter-cluster connectivity. The resulting candidate counts $|\Ec|$ for the four instances are listed in Table~\ref{tab:problem-sizes}. Selected candidate-edge weights are constrained to lie in the box $[w_{\min}, w_{\max}]$, with $w_{\min} = Q_{0.75}$ and $w_{\max} = Q_{0.95}$ the $0.75$ and $0.95$ quantiles of the same lognormal as the base edges. This places new edges in the upper part of the existing weight spectrum, modeling network reinforcement by strong new links rather than weak infill.

Candidate-edge costs follow the linear per-unit deployment model:
\(
\vc = c_0\vone+\beta\,\vd_c .
\)
The distance-independent term $c_0$ is a baseline per-unit cost, while $\beta\,\vd_c$ penalizes Euclidean length. We anchor the baseline unit cost $c_0$ to the median of each instance's base-edge weight distribution. The entries of $\vd_c$ are normalized lengths $d_e/\bar d$, where $d_e$ is the candidate-edge Euclidean length and $\bar d$ is the mean candidate length, making costs comparable across instances. We sweep $\beta\in\{0,0.4,1.6,3.2\}$. At $\beta=0$, all candidates have the same per-unit cost $c_0$, so geometry does not affect the budget. At $\beta=0.4$, length introduces only a mild premium. Larger values, $\beta=1.6$ and $3.2$, increasingly penalize long candidate edges and favor shorter additions. 

\begin{figure}[!t]
    \centering
    \begin{subfigure}{\columnwidth}
        \centering
        \includegraphics[width=\linewidth]{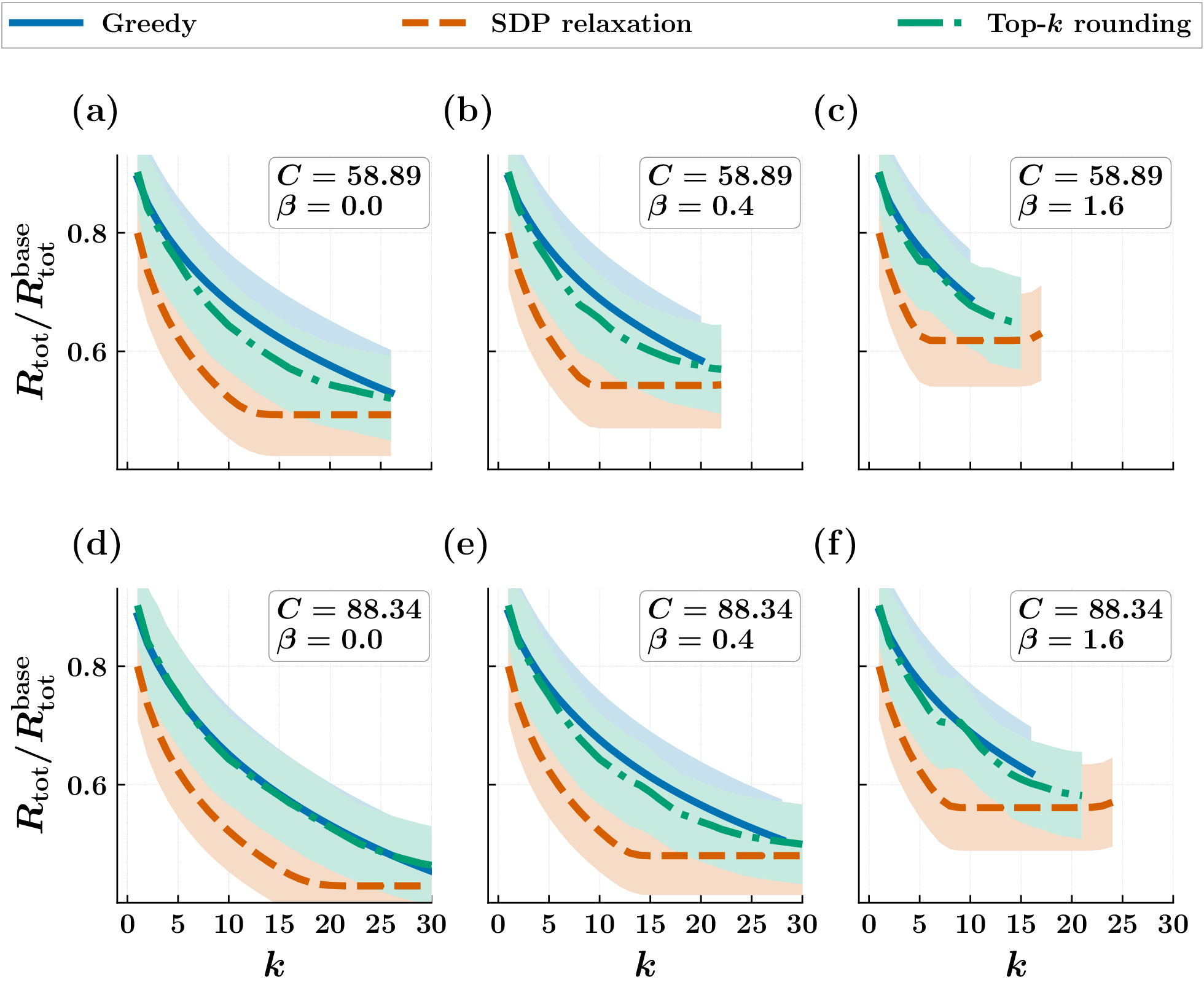}
        \caption{Seed-averaged algorithm comparison (SDP relaxation, top-$k$ rounding, greedy) at $n=50$, $\mu=0$, $\sigma=1$. Curves are the mean over $20$ random base graphs of $R_{\mathrm{tot}}/R_{\mathrm{tot}}^{\mathrm{base}}$; shaded bands span $\pm 1$ standard deviation about the mean across the $20$ seeds. Rows fix the budget $C\in\{58.89, 88.34\}$; columns sweep $\beta$ from weak to strong.}
        \label{fig:r_tot_vs_k}
    \end{subfigure}\\[2pt]
    {\color{gray!50}\hdashrule{\columnwidth}{0.3pt}{2pt}}\\[12pt]
    \begin{subfigure}{\columnwidth}
        \centering
        \includegraphics[width=\linewidth]{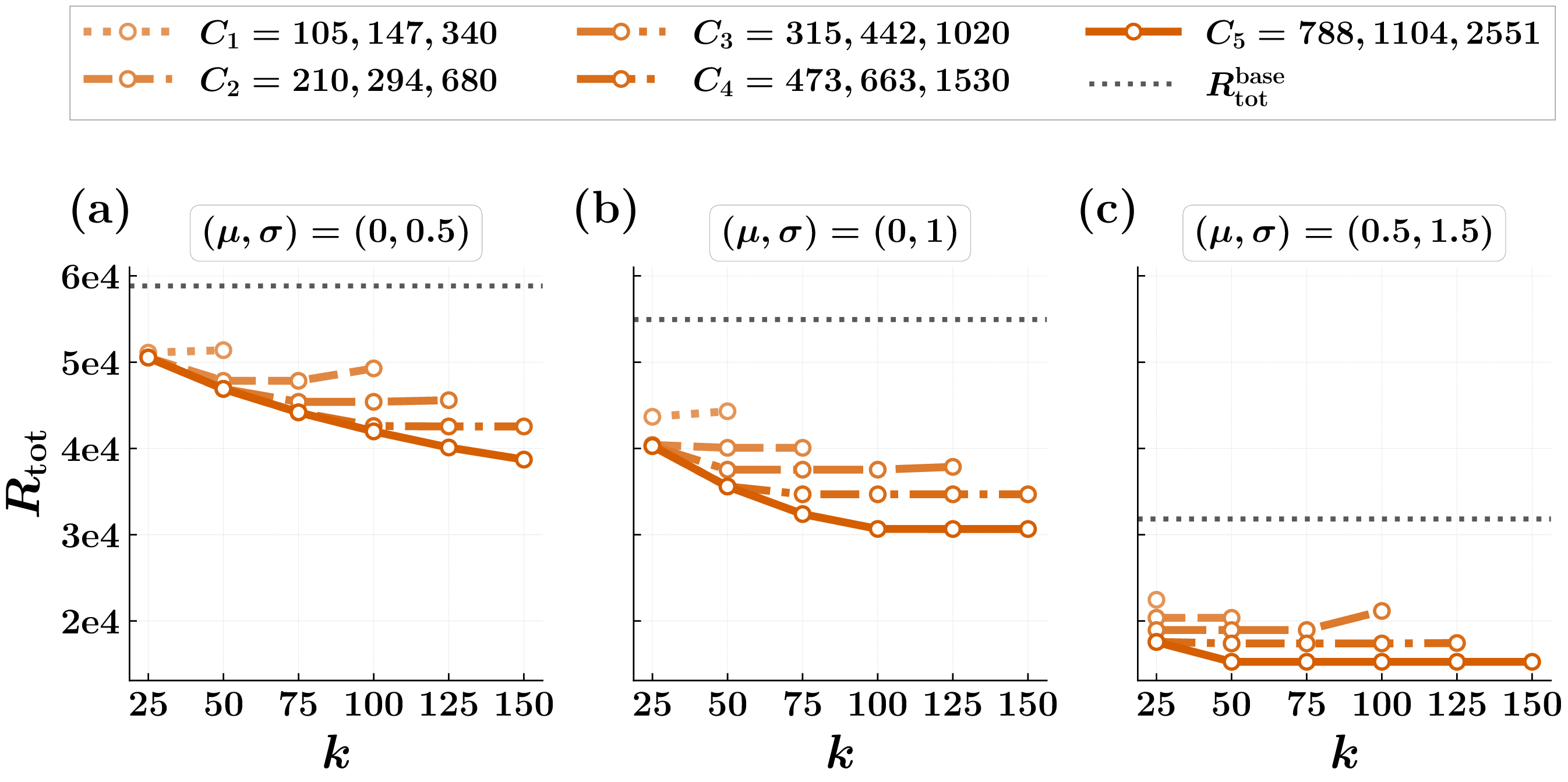}
        \caption{SDP relaxation under varying budgets $C$ at $n=500$, $\beta=1.6$; dotted line is the un-augmented baseline $R_{\mathrm{tot}}^{\mathrm{base}}$ (legend lists $C$).}
        \label{fig:r_tot_vs_k_budget}
    \end{subfigure}
    \caption{$R_{\mathrm{tot}}$ vs the number $k$ of selected augmentation edges.}
    \label{fig:combined_rtot}
\end{figure}

We derive five budget levels $C_1 < \dots < C_5$ from the candidate-edge cost distribution rather than fixing them a priori. The minimum feasible budget for $k$ edges is $C_{\min}(k) = w_{\min}\sum_{i=1}^{k} c_{(i)}$, summed over the $k$ cheapest candidates. The anchor $C_{\min}(k_{\max})$ is the smallest budget that permits all $k_{\max}$ additions. Two levels below the anchor are set to the feasibility budgets for $k\approx\tfrac{1}{3} k_{\max}$ and $\tfrac{2}{3} k_{\max}$ (binding for large $k$, slack for small $k$), and two above at $1.5\times$ and $2.5\times$ the anchor probe the relaxed regime. The tightest budget $C_1$ constrains the problem severely, and the largest budget $C_5$ permits nearly unrestricted additions. The intermediate levels sample the transition between these extremes. All five values depend on the geometry and the lognormal parameters. We compute them once at a fixed nominal cost scale and hold them fixed across the $\beta$ sweep.

\noindent\textit{Real-network case studies.}
We apply the same experimental protocol to two real infrastructure networks: the IEEE 57-bus power grid~\cite{uwpstca,matpower} and the OpenFlights US air-transport network~\cite{openflights}. For the IEEE 57-bus network, we use $n=57$ nodes and assign each existing edge the series-admittance magnitude $1/\sqrt{r^2+x^2}$, computed from the branch resistance $r$ and reactance $x$; schematic bus coordinates are used for distance-based costs. For the OpenFlights network, we take the largest connected component of the busiest US airports ($n=49$), with real airport coordinates and edge weights equal to the number of scheduled routes between airport pairs. In both cases, we keep the clustered $k$-NN candidate-generation rule fixed with $K=k_{\mathrm{in}}=3$. Candidate weights are constrained to the interval $[Q_{0.75},Q_{0.95}]$, where $Q_{0.75}$ and $Q_{0.95}$ are the corresponding quantiles of a lognormal fit to the empirical base-edge weights of each network. The edge-addition cost follows the same model described above.

We evaluate the SDP relaxation, its rounding, and the greedy algorithm on the budgeted effective resistance minimization problem. Each method receives the same inputs at every operating point, so the comparison isolates algorithmic behavior. A feasible integer design is recovered from the fractional SDP solution by the \emph{rounding-and-repair} procedure of Section~\ref{framework}, which we also compare against alternative rounding schemes.

\begin{figure}[!t]
    \centering
    \includegraphics[width=\linewidth]{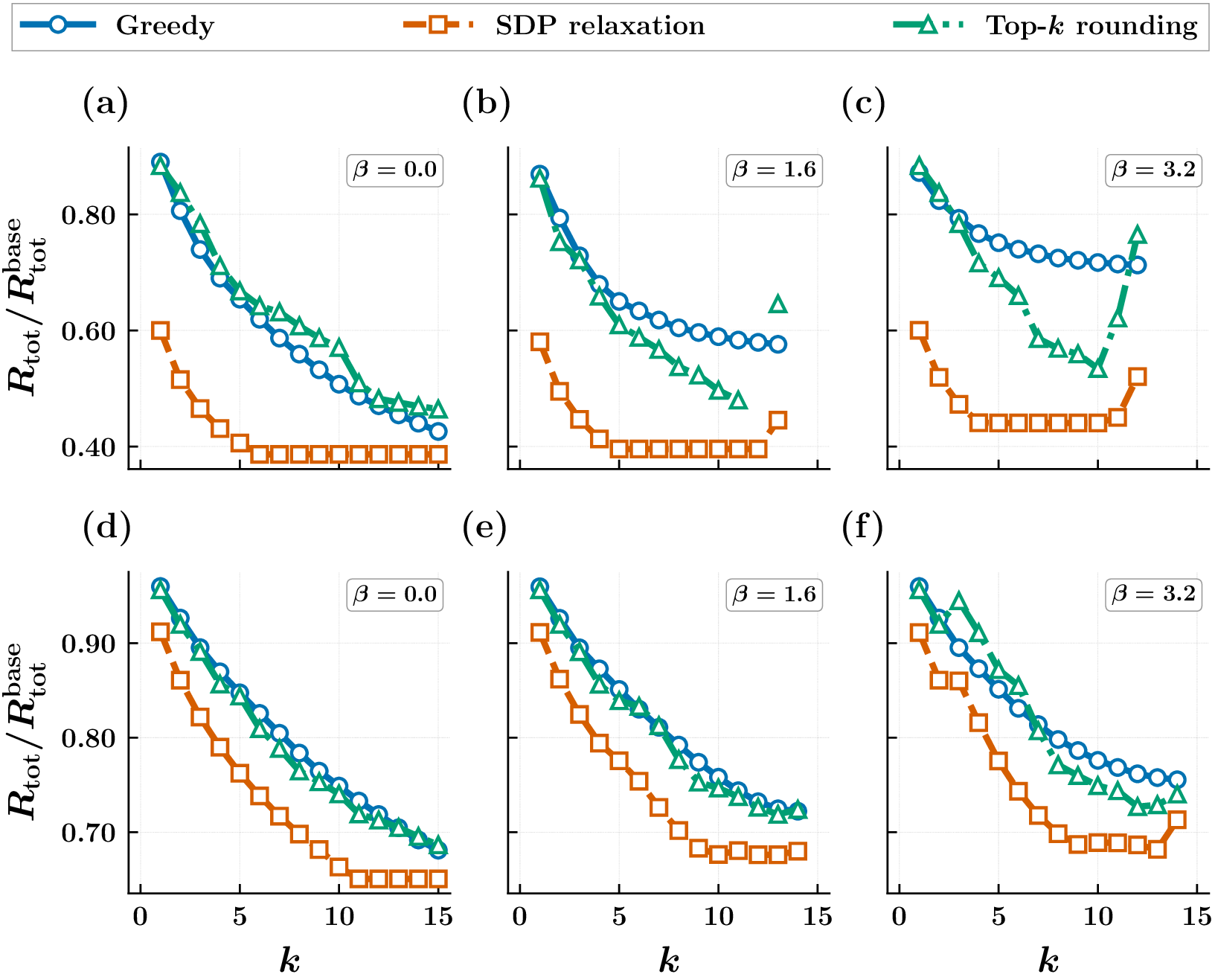}
    \caption{$R_{\mathrm{tot}}/R_{\mathrm{tot}}^{\mathrm{base}}$ versus the number of augmentation edges $k$. Top row: IEEE 57-bus grid; bottom row: OpenFlights air-transport network.}
    \label{fig:realworld}
\end{figure}

\section{Results}
\label{sec:results}

All experiments use the setup of Section~\ref{sec:experimental-setup}. We first report solver behavior, then illustrate the greedy bounds, compare the algorithms, assess rounding quality, and examine the spatial distribution of added-edge weights.

\noindent\textit{Solver behavior.} The HSDE-ADMM solver of
Section~\ref{conic_and_adpt_relax} converges on all four experimental instances.
Algorithm~\ref{alg:hsde}
terminates on the homogeneity-invariant embedding residual $r_t$ of line~\ref{ln:res}. This
residual quantifies the distance of the iterate from the embedding cone and does not certify
the recovered objective value or the primal-dual gap. We therefore also record the recovered primal and dual
objectives $R_{\mathrm{tot}}^{\mathrm{pri}}$ and $R_{\mathrm{tot}}^{\mathrm{dual}}$ at every
iteration with $\tau^t>0$, for the representative case of each instance. At the same iterates
we track the decomposed relative residuals $\varepsilon_{\mathrm{pri}}$,
$\varepsilon_{\mathrm{dual}}$, and $\varepsilon_{\mathrm{gap}}$ of the recovered point
$\bar{\vx}\coloneqq\vx^t/\tau^t$ and $\bar{\vy}\coloneqq\vy^t/\tau^t$. The primal and dual objective values
of $(\mathcal{P}2)$ at the recovered point are $R_{\mathrm{tot}}^{\mathrm{pri}}\coloneqq\vq^\top\bar{\vx}-n$
and $R_{\mathrm{tot}}^{\mathrm{dual}}\coloneqq-\vb^\top\bar{\vy}-n$, with the additive constant $-n$ of
$(\mathcal{P}1)$ restored as in $R_{\mathrm{tot}}^{\mathrm{SDP}}$. Both converge to the relaxation
optimum $\mathrm{val}(\mathcal{P}1)$ as the solver tolerance tightens. The plotted and tabulated
$R_{\mathrm{tot}}^{\mathrm{SDP}}$ throughout the results is this numerical HSDE-ADMM value, which
approximates $\mathrm{val}(\mathcal{P}1)$ to the reported residuals. The recovered dual iterate is
only approximately feasible, so this value is a numerical SDP benchmark rather than a guaranteed
lower bound at finite tolerance. The residuals
$\varepsilon_{\mathrm{pri}}$, $\varepsilon_{\mathrm{dual}}$, and $\varepsilon_{\mathrm{gap}}$ are
the relative primal, dual, and duality-gap residuals of $(\mathcal{P}2)$, given by
$\varepsilon_{\mathrm{pri}}\coloneqq\|\mF\bar{\vx}+\bar{\vs}-\vb\|_2/(1+\|\vb\|_2)$,
$\varepsilon_{\mathrm{dual}}\coloneqq\|\mF^\top\bar{\vy}+\vq\|_2/(1+\|\vq\|_2)$, and
$\varepsilon_{\mathrm{gap}}\coloneqq|\vq^\top\bar{\vx}+\vb^\top\bar{\vy}|
/(1+|\vq^\top\bar{\vx}|+|\vb^\top\bar{\vy}|)$, where $\bar{\vs}\coloneqq\vs^t/\tau^t$ is the recovered slack~\cite{o2016conic}.
Fig.~\ref{fig:solver_convergence}
plots the two objectives and the gap against the iteration, and
Table~\ref{tab:solver_convergence} reports the final relative residuals
$\varepsilon_{\mathrm{pri}}$, $\varepsilon_{\mathrm{dual}}$, and $\varepsilon_{\mathrm{gap}}$,
together with the iteration count and runtime. The stopping tolerance $\varepsilon=10^{-6}$
applies to the embedding residual $r_t$ of line~\ref{ln:res}. The embedding holds $\tau^t=0$
over most of the iterations, so the recovered objectives are defined only after $\tau^t$
becomes positive. The shaded region in each panel marks the iterations with $\tau^t=0$. The primal and dual objectives
agree to within the final gap on the active window. Both $\varepsilon_{\mathrm{dual}}$
and $\varepsilon_{\mathrm{gap}}$ fall below $\varepsilon=10^{-6}$ on all four
instances. The primal residual $\varepsilon_{\mathrm{pri}}$ stays above it, marginally on N$1$ and N$2$ but by more than an order of magnitude on N$3$ and N$4$. The dual objective
$R_{\mathrm{tot}}^{\mathrm{dual}}$ therefore approaches the relaxation optimum to within the
reported relative residuals.

\begin{figure}[!t]
  \centering
  \includegraphics[width=\linewidth]{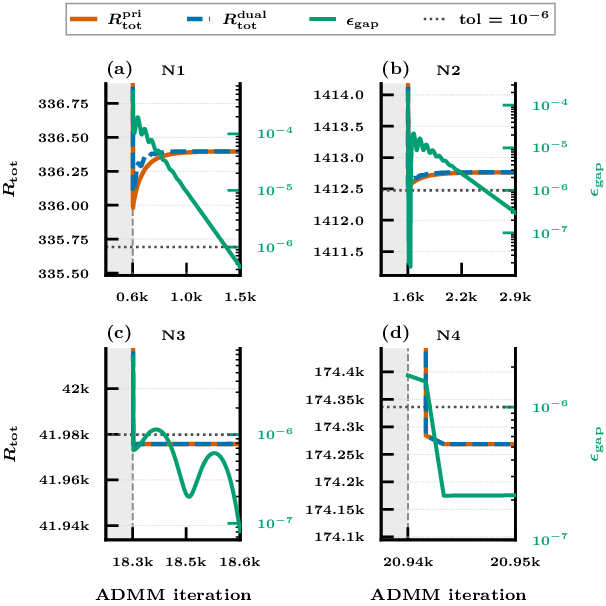}
  \caption{Solver behavior of Algorithm~\ref{alg:hsde} on instances N$1$ to N$4$, with
  $(n,k)$ as $(50,20)$, $(100,40)$, $(500,100)$, and $(1000,200)$. Each panel shows $R_{\mathrm{tot}}^{\mathrm{pri}}$,
  $R_{\mathrm{tot}}^{\mathrm{dual}}$, and $\varepsilon_{\mathrm{gap}}$; the shaded region marks $\tau^t=0$.}
  \label{fig:solver_convergence}
\end{figure}

\begin{table}[!t]
  \centering
  \scriptsize
  \setlength{\tabcolsep}{4pt}
  \renewcommand{\arraystretch}{1.1}
  \caption{Final solver metrics for instances N$1$ to N$4$ at the embedding-residual
  stopping tolerance $\varepsilon=10^{-6}$ (applied to $r_t$).}
  \label{tab:solver_convergence}
  \begin{tabular*}{\columnwidth}{@{\extracolsep{\fill}}lccccc@{}}
    \toprule
    \textbf{Case} & {\boldmath$\varepsilon_{\mathrm{pri}}$} & {\boldmath$\varepsilon_{\mathrm{dual}}$} & {\boldmath$\varepsilon_{\mathrm{gap}}$} & \textbf{iters} & \textbf{runtime (s)} \\
    \midrule
    N$1$ & $1.31{\times}10^{-6}$ & $8.34{\times}10^{-9}$ & $4.50{\times}10^{-7}$ & $1{,}462$  & $1.67$ \\
    N$2$ & $3.08{\times}10^{-6}$ & $2.05{\times}10^{-8}$ & $3.01{\times}10^{-7}$ & $2{,}868$  & $10.59$ \\
    N$3$ & $1.88{\times}10^{-5}$ & $9.12{\times}10^{-8}$ & $8.13{\times}10^{-8}$ & $18{,}618$ & $2{,}060.12$ \\
    N$4$ & $1.34{\times}10^{-5}$ & $1.79{\times}10^{-7}$ & $2.17{\times}10^{-7}$ & $20{,}949$ & $13{,}372.70$ \\
    \bottomrule
  \end{tabular*}
\end{table}

\noindent\textit{Greedy-theorem diagnostics.} We illustrate the two greedy guarantees of
Theorems~\ref{thm:policy_bound_alg2} and~\ref{thm:explicit_rho_alg2}
numerically on representative instances.
Fig.~\ref{fig:greedy_decay} illustrates Theorem~\ref{thm:policy_bound_alg2} on a small geometric instance with \(n=14\), \(p=8\), \(\beta=3.2\), \(\sigma=0.5\), and \(k=4\). The budget is set to \(C\approx26.77\), about \(1.2\times\) the minimum spend needed to fit four edges at weight \(w_{\min}\). Among feasible small instances that use all four additions, we show the one with the largest observed departure of greedy from the Bellman optimum, \(V_1/G_1\approx1.13\). This provides a stress case for visualizing the gap. On easier instances the curves are nearly indistinguishable.
Fig.~\ref{fig:greedy_decay}(a) plots the optimal value-to-go \(V_l\) from exact
backward recursion, the greedy value \(G_l\), and the \(\rho_{\min}\) decay
envelope against the iteration \(l\). The greedy curve stays between zero and the
optimum, illustrating claim~(\ref{cond:greedy_alg2}).
Fig.~\ref{fig:greedy_decay}(b) compares the realized next value
\(V_{l+1}(\mathcal X_{l+1})\) with the contraction ceiling
\((1-\rho_l/k_{\mathrm{rem}}^{(l-1)})\,V_l(\mathcal X_l)\) at each iteration. The
realized value never exceeds the ceiling, illustrating
claim~(\ref{cond:contraction_alg2}).
The second guarantee, the explicit \(\rho_{\min}\) bound of
Theorem~\ref{thm:explicit_rho_alg2}, rests on two spectral envelopes that
Fig.~\ref{fig:spectral_envelopes} illustrates on instance N\(1\) across the three
lognormal regimes. The extreme eigenvalues \(\lambda_{\max}(\mM'_l)\) and
\(\lambda_{\min}(\mM'_l)\) stay inside the constant band
\([\lambda_{\min}^{\mathrm{lower}},\lambda_{\max}^{\mathrm{upper}}]\) as edges are
added, and the per-step gain \(\Delta\Rtot^{(l)}\) inside
\([\delta_{\min},\delta_{\max}]\). Both envelopes defining \(\rho_{\min}\) hold along the plotted greedy trajectories, consistent with Theorem~\ref{thm:explicit_rho_alg2}.

\begin{figure}[!t]
    \centering
    \begin{subfigure}{\columnwidth}
        \centering
        \includegraphics[width=\linewidth]{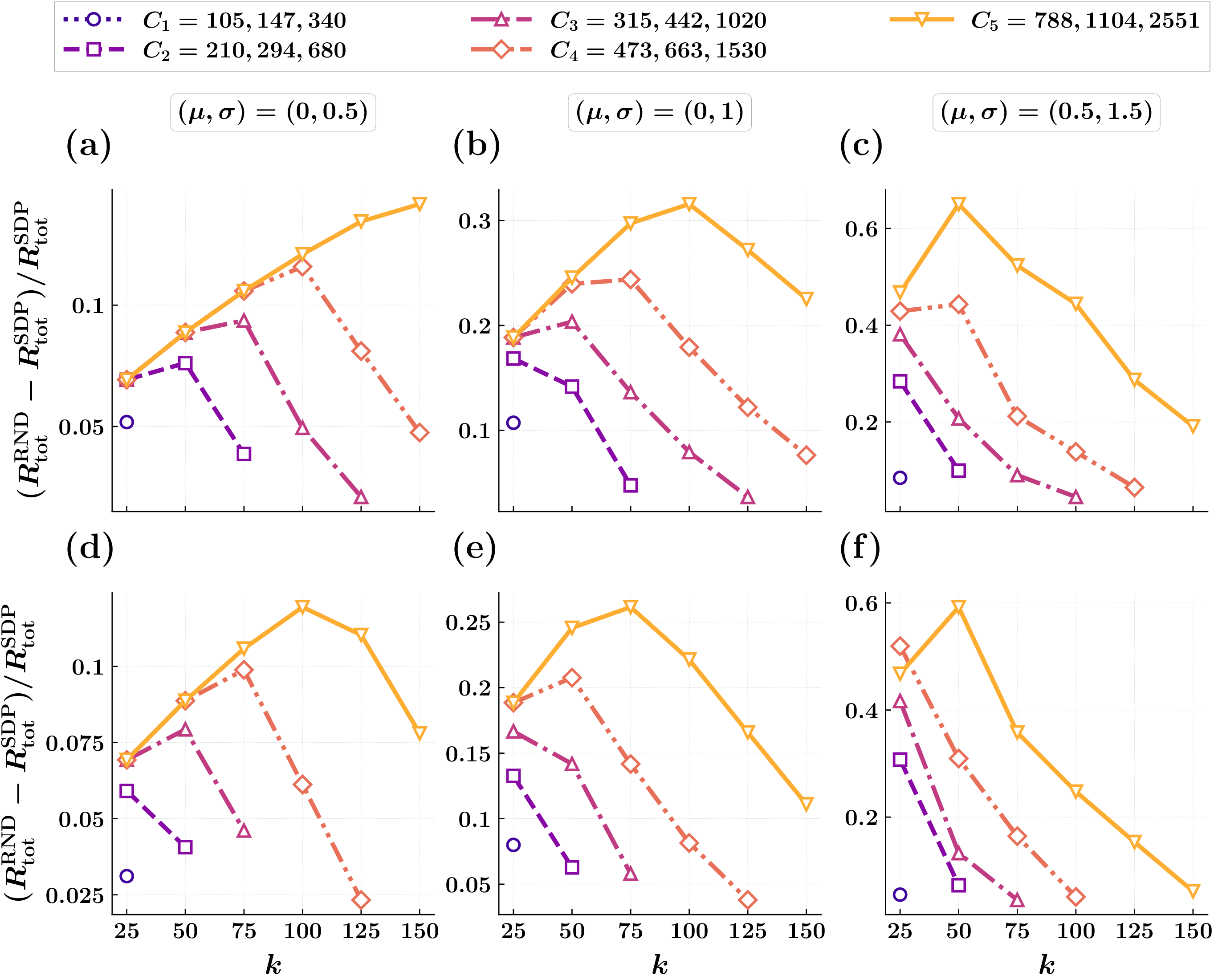}
        \caption{Relative rounding gap to the numerical SDP benchmark $(R_{\mathrm{tot}}^{\mathrm{RND}}-R_{\mathrm{tot}}^{\mathrm{SDP}})/R_{\mathrm{tot}}^{\mathrm{SDP}}$ versus $k$ at $n=500$, for the budget regimes of Section~\ref{sec:experimental-setup}; rows are $\beta\in\{1.6,3.2\}$ with five budget curves per panel.}
        \label{fig:gap}
    \end{subfigure}\\[2pt]
    {\color{gray!50}\hdashrule{\columnwidth}{0.3pt}{2pt}}\\[12pt]
    \begin{subfigure}{\columnwidth}
        \centering
        \includegraphics[width=\linewidth]{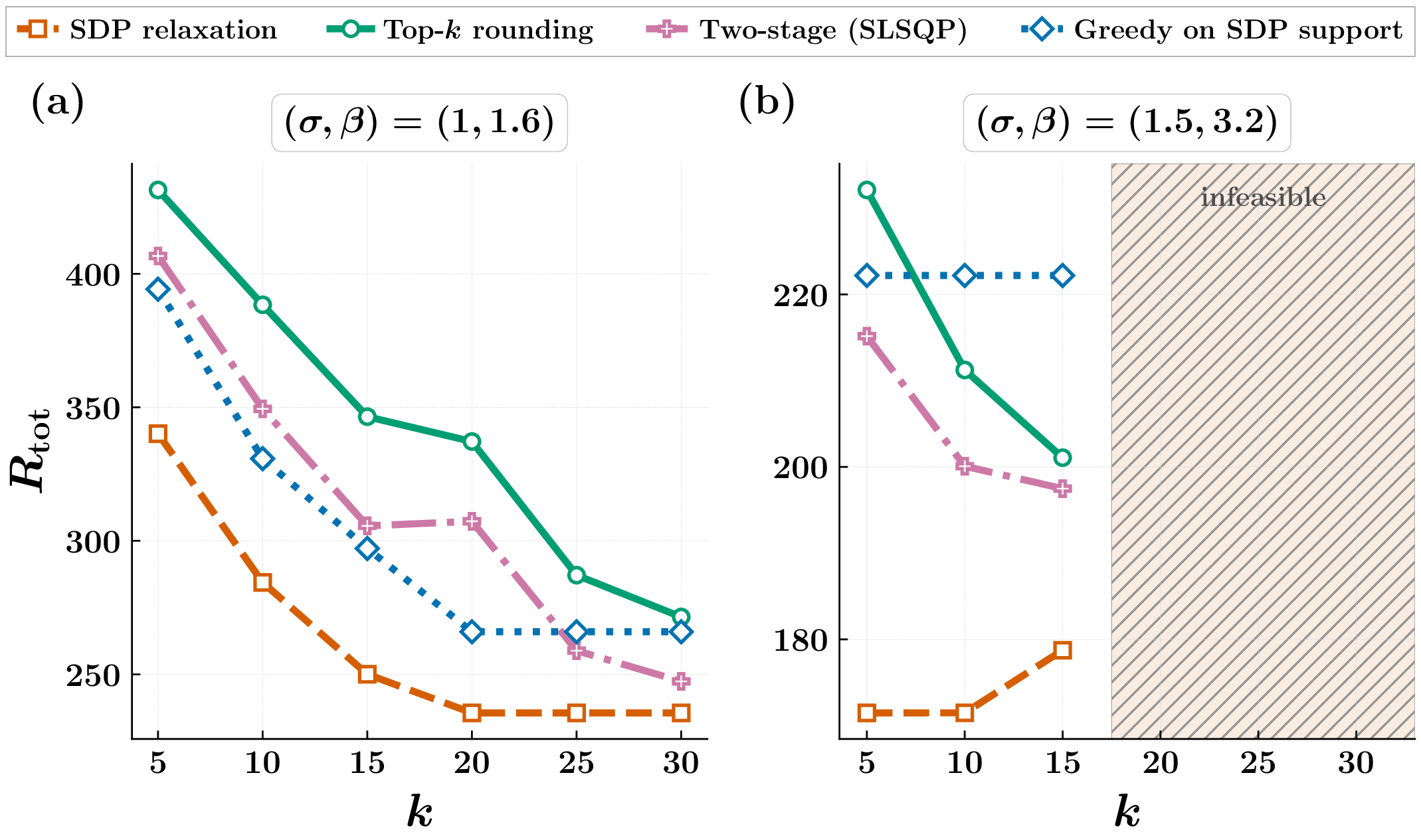}
        \caption{Rounding-scheme comparison at $n=50$: $R_{\mathrm{tot}}$ versus $k$ for a loose-budget (left) and tight-budget (right) regime; the shaded band marks infeasible $k$.}
        \label{fig:rounding_n50}
    \end{subfigure}
    \caption{Rounding quality. (i)~relative gap of top-$k$ rounding to the numerical SDP benchmark; (ii)~the three rounding schemes compared.}
    \label{fig:rounding_combined}
\end{figure}

\noindent\textit{Algorithm comparison.}~We compare the SDP relaxation, its rounding, and the greedy algorithm,
quantifying how much each reduces the total effective resistance
$R_{\mathrm{tot}}$.
Fig.~\ref{fig:combined_rtot}\,(\subref{fig:r_tot_vs_k}) reports the $n=50$ comparison. Each curve is averaged over $20$ random base graphs and normalized per instance by its baseline $R_{\mathrm{tot}}^{\mathrm{base}}$. The shaded region gives $\pm1$ standard deviation. The SDP relaxation is plotted as a numerical SDP benchmark, not as an implementable design. The feasible methods, greedy and rounding, are therefore assessed by their achieved objective and their gap to this benchmark. The gap to the numerical SDP benchmark reflects the combined effect of relaxation, discrete support recovery, numerical solver tolerance, and budget repair. The persistence of the top-$k$ gap suggests that the rounding map and subsequent repair step are important contributors. The exact-$k$ reserve greedy installs all $k$ edges whenever $(C,k)$ is feasible, so its curve spans the budget-feasible range of $k$. Under strong distance-dependent costs this range narrows as $\beta$ grows and the budget tightens, since fitting $k$ edges each at least at $w_{\min}$ becomes costlier. The exact-$k$ greedy must place all $k$ edges and spends myopically, spreading the budget thinner than the top-$k$ rounding. Under distance-dependent costs it therefore trails rounding at larger $\beta$.

The SDP curve is mildly non-monotone in $k$ under the strongest penalty, edging upward at the largest $k$. This does not contradict the weight-monotonicity of $R_{\mathrm{tot}}$. For a fixed support, increasing edge weights can only decrease $R_{\mathrm{tot}}$. The exactly-$k$ constraint changes the support as $k$ grows, while the fixed budget forces weight to be spread across more selected edges. In Fig.~\ref{fig:combined_rtot}\,(\subref{fig:r_tot_vs_k}), relaxing the budget from $C=58.89$ to $C=88.34$ lowers the flattening level from about $0.5$ to near $0.44$ in normalized units.

Fig.~\ref{fig:combined_rtot}\,(\subref{fig:r_tot_vs_k_budget}) extends the SDP budget sweep to $n=500$. This panel reports raw $R_{\mathrm{tot}}$ values with a different budget scale, and shows that the same budget-saturation trend appears in the relaxed problem at larger scale.

\noindent\textit{Real-network case studies.}~Fig.~\ref{fig:realworld} reports the reduction in $R_{\mathrm{tot}}$ for edge augmentation on the real infrastructure networks. For each network, the budget is fixed to its anchor value $C_{\min}(k_{\max})$, with $k_{\max}=15$, and the distance-cost parameter is varied over $\beta\in\{0,1.6,3.2\}$. The numerical SDP benchmark lies below both feasible methods across both networks and all $\beta$. For $\beta=0$, shown in Fig.~\ref{fig:realworld}(a) and Fig.~\ref{fig:realworld}(d), the greedy and rounding solutions are nearly indistinguishable. For larger values of $\beta$, the two feasible methods separate. The difference is most pronounced on the IEEE 57-bus grid at $\beta=3.2$, where Fig.~\ref{fig:realworld}(c) shows rounding below greedy across intermediate values of $k$. The same ordering appears on the OpenFlights network in Fig.~\ref{fig:realworld}(f), although the gap is smaller.

\noindent\textit{Rounding quality.}~Fig.~\ref{fig:rounding_combined}\,(\subref{fig:gap}) measures the relative gap of the integer-feasible rounded solution to the numerical SDP benchmark. The gap is itself non-monotone in $k$. It grows from near zero at small $k$, peaks at a regime-dependent cardinality $k^{*}$, and then decays in the saturation regime as the rounded objective moves closer to the numerical SDP benchmark. The peak moves to lower $k$ as the budget $C$ tightens, and its amplitude grows with $\sigma$. The rounded solution is therefore closest to the numerical SDP benchmark in the small-$k$ and saturated regimes. The worst-case penalty is concentrated at intermediate $k$ under heavy-tailed weights.
This motivates benchmarking three rounding schemes for recovering an integer design from the relaxed solution. \emph{Top-$k$} is the
\emph{rounding-and-repair} baseline of Section~\ref{framework}, the standard
relax-and-select heuristic~\cite{joshi2008sensor}. It keeps the $k$ edges of
largest relaxed selection $\vz^\star$, clamps their weights to
$[w_{\min},w_{\max}]$, and scales to the budget. \emph{Two-stage} keeps that support but
re-optimizes its weights via SLSQP, a convex subproblem with a closed-form
gradient~\cite{Minimizing_ER} that is no worse than top-$k$. \emph{Greedy on SDP
support} instead gates the candidates to the relaxation support, the edges whose
relaxed weight exceeds half its maximum. It adds edges by largest marginal
$R_{\mathrm{tot}}$ gain, drawing from all candidates when the support holds
fewer than $k$. Across both regimes (Fig.~\ref{fig:rounding_combined}\,(\subref{fig:rounding_n50})) the numerical SDP
benchmark lies below every rounding, two-stage dominates top-$k$, and greedy on the SDP support is
competitive at small $k$.

\begin{figure}[!t]
    \centering
    \begin{subfigure}{\columnwidth}
        \centering
        \includegraphics[width=\linewidth]{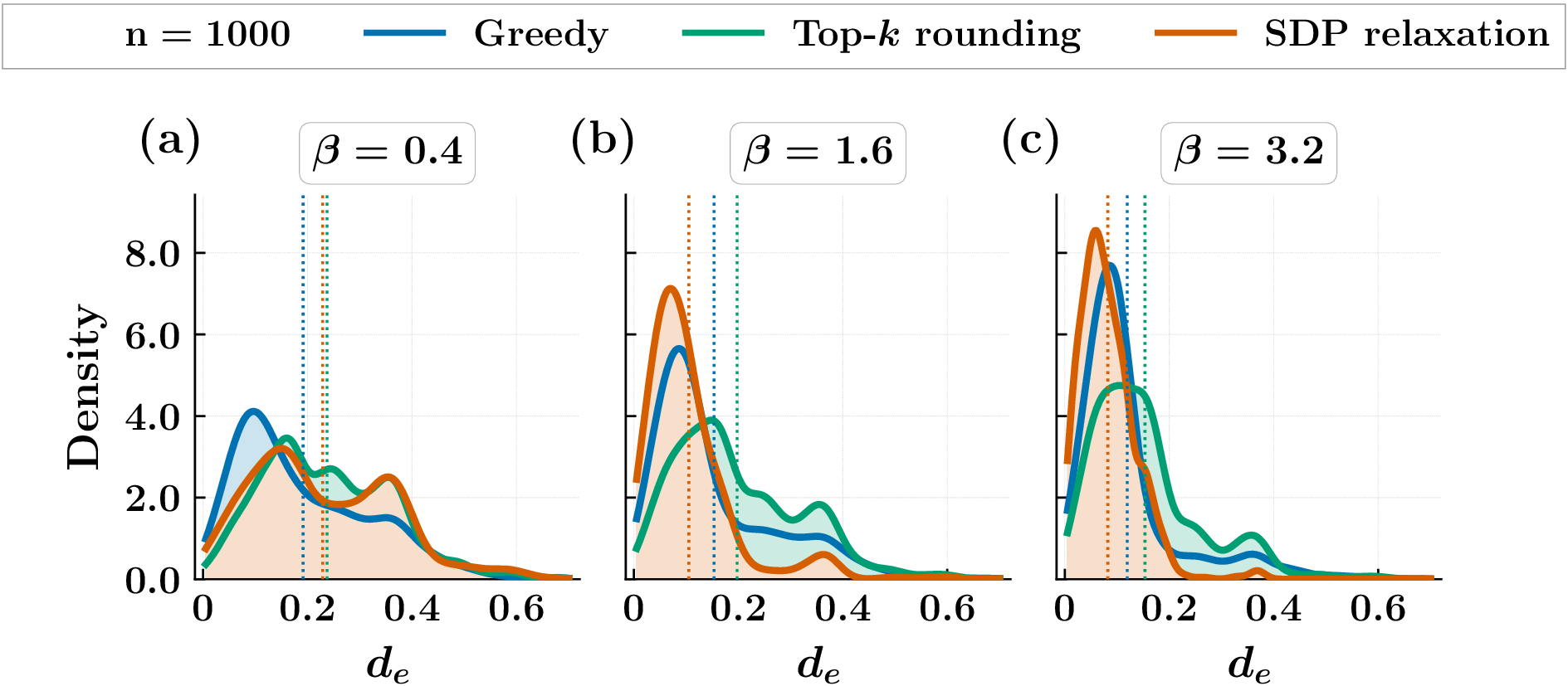}
        \caption{$n=1000$, $\mu=0$, $\sigma=0.5$; panels sweep distance-penalty $\beta$.}
        \label{fig:distance_distribution_grid_1000}
    \end{subfigure}\\[2pt]
    {\color{gray!50}\hdashrule{\columnwidth}{0.3pt}{2pt}}\\[12pt]
    \begin{subfigure}{\columnwidth}
        \centering
        \includegraphics[width=\linewidth]{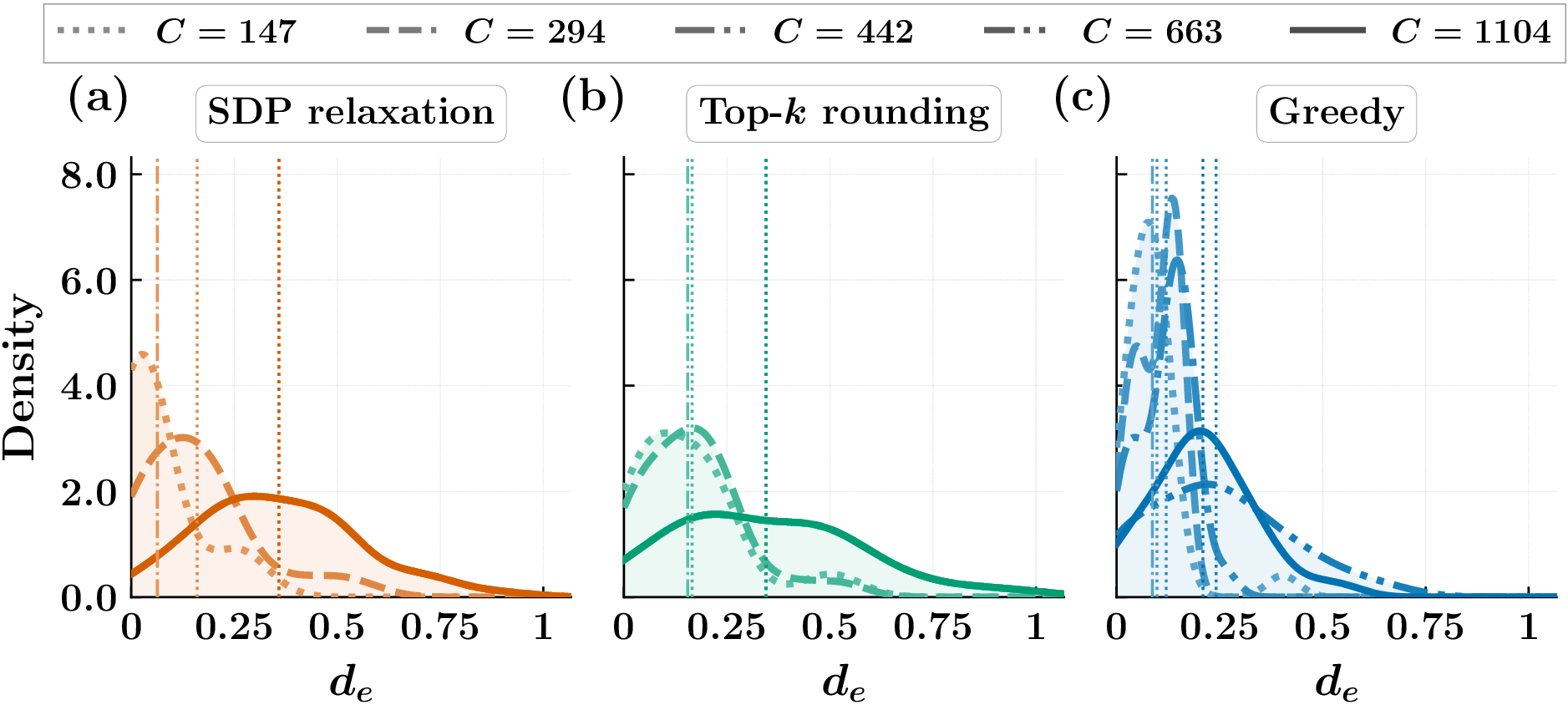}
        \caption{Fixed setting $(n=500,\ \mu=0,\ \sigma=1,\ \beta=1.6,\ k=25)$; curves sweep cost budget $C$.}
        \label{fig:Cost_budget_sensitivity}
    \end{subfigure}
    \caption{Gaussian KDE of selected-edge Euclidean distances for the SDP relaxation, its rounding, and the greedy algorithm. Dotted vertical lines mark the weighted mean distance per curve.}
    \label{fig:combined_kde}
\end{figure}

\noindent\textit{Spatial footprint.}
As a secondary diagnostic, Fig.~\ref{fig:combined_kde} summarizes how the selected conductance is distributed over candidate-edge lengths. The $x$-axis is the Euclidean endpoint separation \(d_e\), and the curves are weighted Gaussian kernel density estimates (KDEs). Greedy and rounding assign each selected edge a KDE weight equal to its installed conductance \(w_e\). The SDP relaxation instead uses the relaxed weights \(w_e^\star\), restricted to dominant relaxed selections \(w_e^\star>\tfrac12\max_{e'\in\Ec}w_{e'}^\star\). Thus the curves show the distribution of installed or relaxed conductance over edge length, not merely the number of selected edges. Dotted vertical lines mark the weighted mean distance.
  Fig.~\ref{fig:combined_kde} separates the algorithms along two axes. Panel~(\subref{fig:distance_distribution_grid_1000}) sweeps the distance penalty $\beta$ at $n=1000$. The rounding and SDP KDEs nearly coincide at $\beta=0.4$, and greedy is already mildly concentrated on shorter edges. The three curves separate further as $\beta$ grows to $1.6$ and $3.2$. The SDP concentrates sharply on short edges, while greedy retains a broader tail from occasional selection of longer edges with large marginal resistance reduction. Panel~(\subref{fig:Cost_budget_sensitivity}) varies the budget $C$. The SDP saturates for $C \ge 442$. Tight budgets pin its weights on short, cheap edges, while larger budgets reallocate toward longer, higher-gain edges until the budget becomes nonbinding. Its curves for $C \in \{442,663,1104\}$ therefore coincide, and rounding inherits the collapse, leaving only $C=147$ and $294$ distinct. Greedy instead does not saturate, so all five budget curves remain distinct as the relaxed ceiling reorders its reserve-capped sequential selection.

\section{Conclusion}

This paper studied budget-constrained augmentation of a pre-deployed weighted graph with total effective resistance as the robustness objective. The framework combines an SDP relaxation, feasible rounding procedures, and a cost-aware greedy method built on rank-one updates and biharmonic-distance caching. The relaxation lower-bounds the mixed-integer optimum and therefore supplies an \emph{a posteriori} gap estimate against any feasible discrete design. Empirically, distance-dependent per-unit costs reshape both the attainable resistance reduction and the spatial layout of installed conductance. They also expose a weakness of naive top-$k$ rounding: it turns a fractional, cost-aware SDP allocation into a discrete support and only then repairs the weights, whereas re-optimizing weights on that support recovers much of the lost performance. For the greedy method, the Bellman analysis yields a local residual-decay interpretation and a conservative spectral floor on the local policy ratio. Overall, the experiments show that cost-aware augmentation behaves differently from uniform-cost edge installation, both in the magnitude of robustness improvement and in the geometry of the selected links.

\appendices

\section{}\label{app:policy_bound}
\begin{proof}[Proof of Theorem~\ref{thm:policy_bound_alg2}]
The greedy edge is \(e_l^\star=\arg\max_{e\in\mathcal F_l}\Delta R_{\mathrm{tot}}^{(l)}(e)\), so \(\Delta R_{\mathrm{tot}}^{(l)}(e_l^\star)=\max_{e\in\mathcal F_l}\Delta R_{\mathrm{tot}}^{(l)}(e)\); Definition~\ref{def:local_policy_ratio_alg2} bounds this maximum below by \(\tfrac{\rho_l}{k_{\mathrm{rem}}^{(l-1)}}\,V_l(\mathcal X_l)\), which is~(\ref{cond:progress_alg2}).
Next, by the Bellman recursion, evaluating the maximum total achievable gain at the feasible greedy edge \(e_l^\star\) yields
\[
V_l(\mathcal X_l)
\ge
\Delta R_{\mathrm{tot}}^{(l)}(e_l^\star)
+
V_{l+1}\!\bigl(f(\mathcal X_l,e_l^\star)\bigr).
\]
Thus, after re-arrangement and using~(\ref{cond:progress_alg2}), this proves~(\ref{cond:contraction_alg2}).

For the value bound we argue by backward induction. The base case sets \(G_l(\mathcal X_l)=V_l(\mathcal X_l)=0\) when \(\mathcal F_l=\varnothing\) or \(k_{\mathrm{rem}}^{(l-1)}=0\); otherwise \(\mathcal F_l\neq\varnothing\) and \(k_{\mathrm{rem}}^{(l-1)}\ge1\), and
\[
\begin{aligned}
G_l(\mathcal X_l)
&= \Delta R_{\mathrm{tot}}^{(l)}(e_l^\star) + G_{l+1}\!\bigl(f(\mathcal X_l,e_l^\star)\bigr)\\
&\le \Delta R_{\mathrm{tot}}^{(l)}(e_l^\star) + V_{l+1}\!\bigl(f(\mathcal X_l,e_l^\star)\bigr)
\le V_l(\mathcal X_l),
\end{aligned}
\]
where the first inequality is the induction hypothesis \(G_{l+1}\le V_{l+1}\) and the second the Bellman optimality of \(V_l(\mathcal X_l)\) over \(\mathcal F_l\ni e_l^\star\).
Each one-step gain \(\Delta R_{\mathrm{tot}}^{(l)}(e)\ge 0\) (nonnegative numerator, denominator at least one), so the greedy value \(G_l(\mathcal X_l)\), a sum of such gains, satisfies \(G_l(\mathcal X_l)\ge 0\); together with \(G_l(\mathcal X_l)\le V_l(\mathcal X_l)\) this establishes~(\ref{cond:greedy_alg2}). Finally, \eqref{eq:product_decay} follows by applying~(\ref{cond:contraction_alg2}) at each step \(t=l,\dots,m\) along the greedy trajectory and multiplying the resulting contraction factors. The inequality \(1-z\le e^{-z}\) is used to replace the exact multiplicative residual-decay product by an exponential upper bound, thereby yielding a more interpretable closed-form estimate. The bound \eqref{eq:product_decay_uniform} follows by substituting
\(k_{\mathrm{rem}}^{(t-1)}=k-(t-1)\), bounding \(\rho_t\ge\rho_{\min}\) in each factor, re-indexing with \(j=k-t+1\), and using
\(1-z\le e^{-z}\) for \(z\ge 0\); the harmonic sum in the exponent is
\(\sum_{t=l}^{m}\tfrac{1}{k-t+1}=\sum_{j=k-m+1}^{k-l+1}\tfrac1j=H_{k-l+1}-H_{k-m}\).
\end{proof}

\section{}\label{app:explicit_rho}
\begin{proof}[Proof of Theorem~\ref{thm:explicit_rho_alg2}]
Since the rank-one updates are positive semidefinite, no candidate edge is augmented twice, and the selected weights lie in
\([w_{\min},w_{\max}]\), every realized greedy iterate satisfies \( \mM'_0 \preceq \mM'_l \preceq \mM'_{\max},
\, \text{for} \; l=1,\dots,k.\)
Therefore,
\(
\lambda_{\min}(\mM'_l)\ge \lambda_{\min}^{\mathrm{lower}},
\, \text{and} \,
\lambda_{\max}(\mM'_l)\le \lambda_{\max}^{\mathrm{upper}}.
\)

Let \(r_e^{(l)} \coloneqq \va_e^\top(\mM'_{l-1})^{-1}\va_e\) and \(\bigl(b_e^2\bigr)^{(l)} \coloneqq \|(\mM'_{l-1})^{-1}\va_e\|_2^2\) be the effective resistance and squared biharmonic distance of edge \(e\) at state \(\mathcal X_l\). The Rayleigh-quotient bounds for \((\mM'_{l-1})^{-1}\) and \((\mM'_{l-1})^{-2}\), with \(\|\va_e\|_2^2=2\) and the envelopes above, give \(r_e^{(l)} \le 2/\lambda_{\min}^{\mathrm{lower}}\) and \(2/(\lambda_{\max}^{\mathrm{upper}})^2 \le (b_e^2)^{(l)} \le 2/(\lambda_{\min}^{\mathrm{lower}})^2\).

Hence, for every feasible edge \(e\in\mathcal F_l\), the gain \(\Delta R_{\mathrm{tot}}^{(l)}(e)=n\,w_e^{(l)}\bigl(b_e^2\bigr)^{(l)}/\bigl(1+w_e^{(l)}r_e^{(l)}\bigr)\) is increasing in \(w_e^{(l)}\). Its lower bound follows by taking \(w_{\min}\) in the numerator, \(w_{\max}\) in the denominator, \(\bigl(b_e^2\bigr)^{(l)}\ge 2/(\lambda_{\max}^{\mathrm{upper}})^2\), and \(r_e^{(l)}\le 2/\lambda_{\min}^{\mathrm{lower}}\); its upper bound follows from \(1+w_e^{(l)}r_e^{(l)}\ge 1\), \(w_e^{(l)}\le w_{\max}\), and \(\bigl(b_e^2\bigr)^{(l)}\le 2/(\lambda_{\min}^{\mathrm{lower}})^2\). Together,
\[
\begin{aligned}
\delta_{\min}
\;\coloneqq\;
\frac{2n\,w_{\min}/(\lambda_{\max}^{\mathrm{upper}})^2}{1+2w_{\max}/\lambda_{\min}^{\mathrm{lower}}}
&\;\le\;
\Delta R_{\mathrm{tot}}^{(l)}(e)\\
&\;\le\;
\frac{2n\,w_{\max}}{(\lambda_{\min}^{\mathrm{lower}})^2}
\;\eqqcolon\;
\delta_{\max}.
\end{aligned}
\]

Since \(\delta_{\max}\) upper-bounds the one-step gain at every reachable state and at most \(k_{\mathrm{rem}}^{(l-1)}\) augmentations remain from iteration \(l\), we have \(V_l(\mathcal X_l)\le k_{\mathrm{rem}}^{(l-1)}\,\delta_{\max}\). As the best feasible one-step gain is at least \(\delta_{\min}\), it follows that
\[
\max_{e\in\mathcal F_l}\Delta R_{\mathrm{tot}}^{(l)}(e)
\ge
\delta_{\min}
\ge
\dfrac{\delta_{\min}}{\delta_{\max}}\,
\frac{1}{k_{\mathrm{rem}}^{(l-1)}}\,V_l(\mathcal X_l).
\]
Thus, whenever $\rho_l$ is defined,
\(
\rho_l \ge \frac{\delta_{\min}}{\delta_{\max}},
\, l=1,\dots,k,
\)
and substituting the expressions for \(\delta_{\min}\) and \(\delta_{\max}\)
gives \eqref{eq:explicit_rho_alg2}.
\end{proof}

\ifCLASSOPTIONcaptionsoff
  \newpage
\fi

\bibliographystyle{IEEEtran}
\bibliography{reference}

\end{document}